\documentclass[a4paper,leqno,12pt]{amsart}
\usepackage{amssymb,amscd}
\usepackage[latin1]{inputenc}
\usepackage{graphicx}
\usepackage{fullpage}
 \divide\oddsidemargin by 2
\usepackage{enumerate}
\makeatletter
\let\@listlla\list

\def\list#1#2{\@listlla{#1}{#2\itemsep=2pt\parsep=0pt\topsep=3pt plus 1pt minus 1 pt}}
\newcommand{\tridiagram}[6]{{\par\par \centering
\@picture(120,120)(0,0) \put(30,95){\makebox(0,0)[r]{$#1$}}
\put(90,30){\makebox(0,0)[tl]{$#3$}}
\put(90,95){\makebox(0,0)[l]{$#2$}}
\put(60,102){\makebox(0,0)[b]{$#4$}}
\put(102,60){\makebox(0,0)[l]{$#6$}}
\put(50,50){\makebox(0,0)[tr]{$#5$}} \thinlines
\put(40,95){\vector(1,0){40}} \put(95,80){\vector(0,-1){40}}
\put(25,80){\vector(1,-1){55}}
\endpicture\par\par}\noindent\ignorespaces}
\def\@map#1#2[#3]{\mbox{$#1 \colon #2 \longrightarrow #3$}}
\def\map#1#2{\@ifnextchar [{\@map{#1}{#2}}{\@map{#1}{#2}[#2]}}
\newcommand{\Aut}[1]{\mbox{\rm Aut}{(#1)}}

\newcommand{\Ker}[1]{\mbox{\rm Ker}(#1)}

\newcommand{\mc}{\mathbb{C}}
\newcommand{\mz}{\mathbb{Z}}

\newcommand{\mn}{\mathbb{N}}

\newcommand{\F}[1]{\ensuremath{\mathbb{F}_{#1}}}

\makeatother
\newtheorem{theorem}{Theorem}[section]
\newtheorem{theo}{Theorem}[]

\newtheorem{defi}{Definition}[]
\newtheorem{proposition}[theorem]{Proposition}
\newtheorem{corollary}[theorem]{Corollary}
\newtheorem{lemma}[theorem]{Lemma}
\theoremstyle{definition}
\newtheorem{definition}[theorem]{Definition}

\newtheorem{remark}[theorem]{Remark}

\newtheorem{claim}{Claim}

\usepackage{color}

\title{A non-trivial example of a free-by-free group with the Haagerup property}

\author{Fran\c{c}ois Gautero}
\date{\today}
\address{Fran\c{c}ois Gautero, Universit\'e de Nice Sophia Antipolis,
Laboratoire de Math\'ematiques J.A.~Dieudonn\'e (UMR CNRS 6621), Parc Valrose, 06108
Nice Cedex 2, France} \email{Francois.Gautero@unice.fr}


\begin{document}
\begin{abstract}
The aim of this note is to prove that the group of Formanek-Procesi acts properly isometrically on a finite dimensional CAT(0) cube complex. This gives a first example of a non-linear semidirect product between two non abelian free groups which satisfies the Haagerup property.
\end{abstract}

\maketitle

\section*{Introduction}
The Haagerup property is an analytical property introduced in \cite{Haagerup}, where it was proved to hold for free groups:

\begin{defi}[\cite{Haagerup,Cherix}] \hfill

A {\em conditionally negative definite function} on a discrete group $G$ is a function $f \colon G \rightarrow \mc$ such that for any natural integer $n$, for any $\lambda_1, \cdots, \lambda_n \in \mc$  with $\displaystyle \sum^n_{i=1} \lambda_i = 0$, for any $g_1,\cdots,g_n$ in $G$ one has $$\sum_{i,j} \overline{\lambda}_i \lambda_j f(g^{-1}_i g_j) \leq 0.$$

The group $G$ {\em satisfies the Haagerup property}, or {\em is an a-T-menable group}, if and only if there exists a proper conditionnally negative definite function on $G$.
\end{defi}

Groups with the Haagerup property encompass the class of amenable groups, but form a much wider class. Free groups were in some sense the ``simplest'' non-amenable groups with the Haagerup property. Haagerup property has later been renewed by the work of Gromov, where it appeared under the term of a-T-menability. It is now most easily presented as a strong converse to the famous Kazhdan's Property (T) in the sense that a group satisfies both the Haagerup and (T) properties if and only if it is a compact group (a finite group in the discrete case). We refer the reader to \cite{Cherix} for a detailed background and history of Haagerup property.

What do we know about extensions of a-T-menable groups ? By \cite{Jolissaint} a semidirect product of an a-T-menable group with an amenable one is a-T-menable. For instance any semidirect product $\F{n} \rtimes \mz$, where $\F{n}$ denotes the rank $n$ free group, is a-T-menable. Also it has been proved recently that any wreath-product $\F{n} \wr \F{k}$ is a-T-menable (this is a particular case of the various theorems in \cite{Yves} - see also the preliminary paper \cite{Yves0}). But such a result does not hold anymore when considering arbitrary Haagerup-by-Haagerup groups. The most famous counter-example is given by \cite{Burger}: for any free subgroup $\F{k}$ of $\mathrm{SL}_2(\mz)$ the semidirect product
$\mz^2 \rtimes \F{k}$ satisfies a relative version of Kazhdans's Property (T) and thus is not Haagerup (see also \cite{delaHarpeValette} for the relative property (T) of $\mz^2 \rtimes \mathrm{SL}(2,\mz)$ and pass to a finite index free subgroup of $\mathrm{SL}_2(\mz)$ - since the Haagerup property holds for a group $G$ if and only if it holds for a finite index subgroup of $G$, this gives an example as announced).

To what extent can this result be generalized to semidirect products $\F{n} \rtimes \F{k}$ with both $n$ and $k$ greater or equal to $2$ ? These semidirect products lie in some ``philosophical'' sense just ``above'' the groups $\mz^2 \rtimes \F{k}$ (substitute the amenable group $\mz^2$ by the free group $\F{n}$, the simplest example of an a-T-menable but not amenable group) but also just above the groups $\F{n} \rtimes \mz$ (substitute the free abelian group $\mz$ by the free non-abelian group $\F{k}$). The former analogy might lead to think that no group $\F{k} \rtimes \F{n}$ ($n, k \geq 2$) satisfies the Haagerup property whereas the latter one might lead to think that any such group is an a-T-menable group.  The purpose of this short paper is to present a first example of a non-linear a-T-menable semidirect product $\F{n} \rtimes \F{k}$ ($n, k \geq 2$). More precisely:

\begin{defi} \hfill

Let $n$ be any integer greater or equal to $2$. The {\em ${n}^{\mathrm{th}}$-group of Formanek - Procesi} is the semidirect product $\F{n+1} \rtimes_\sigma \F{n}$ where
$\F{n} = \langle t_1,\cdots,t_n \rangle$
and $\F{n+1} = \langle x_1,\cdots,x_{n},y \rangle$ are the rank $n$ and rank $n+1$ free groups and
$\sigma \colon \F{n} \hookrightarrow \Aut{\F{n+1}}$ is the
monomorphism defined as follows:

For $i,j \in \{1,\cdots,n\}$, $\sigma(t_i)(x_j) = x_j$
and $\sigma(t_i)(y) = y x_i$.
\end{defi}

As claimed by this definition, it is easily checked that $\sigma$ is a monomorphism. These groups were introduced in \cite{Formanek} to prove that $\Aut{\F{n}}$ is non linear for $n \geq 3$.

\begin{theo} \hfill
\label{the theorem}

The ${n}^{\mathrm{th}}$-group of Formanek - Procesi acts properly isometrically on some $(2n+2)$-dimensional CAT(0) cube complex and in particular satisfies the Haagerup property.
\end{theo}

Let us briefly recall that a {\em cube complex} is a metric polyhedral
complex in which each cell is isomorphic to the Euclidean cube $[0, 1]^n$
and the gluing maps are isometries. A cube complex is called {\em CAT(0)}
if the metric induced by the Euclidean metric on the cubes turns it
into a CAT(0) metric space (see \cite{Bridson}).
In order to get the above statement, we prove the existence of a
``space with walls'' structure as introduced by Haglund and Paulin \cite{HaglundPaulin}. A theorem of Chatterji-Niblo \cite{ChatterjiNiblo}
or Nica \cite{Nica} (for similar constructions in other settings, see also \cite{NibloRoller} or \cite{Sageev}) gives the announced action on a CAT(0) cube complex.

Even if the proof is quite simple, we think however that our example is worth of interest:
thanks to the profound structure theorem of \cite{BFHpolynomial} about subgroups of polynomially growing automorphisms (our example is a subgroup of linearly growing automorphisms), a more elaborated version of the construction presented here should hopefully lead to a positive answer to the following question, well-known among experts in the field:

\medskip

\noindent {\em Question (folklore): Does any semidirect product $\F{n} \rtimes \F{k}$ over a free subgroup of polynomially growing outer automorphisms satisfy the Haagerup property ?}

\medskip

We guess in fact that any semidirect product $\F{n} \rtimes_\sigma \F{k}$ with $\sigma(\F{k})$ a free subgroup of unipotent polynomially growing outer automorphisms acts properly isometrically on some finite dimensional CAT(0) cube complex the dimension of which depends on the way strata interleave with each other, see the brief discussion at the end of the paper. Since any subgroup of polynomially growing automorphisms admits a unipotent one as a finite-index subgroup \cite{BFHpolynomial}, this would imply a positive answer to the above question.

\section{Preliminaries}

\subsection{Notations}

We will prove Theorem \ref{the theorem} with $n=2$. The reader will easily generalize the construction to any integer $n \geq 2$. With the notations of Theorem \ref{the theorem}, the group $G := \F{3} \rtimes_\sigma \F{2}$ admits $$\langle x_i,y,t_j \mbox{ ; } t^{-1}_j x_i t_j = x_i \mbox{, } t^{-1}_j y t_j = y x_j \mbox{, } i,j=1,2 \rangle$$ as a presentation. We denote by $S$ the generating set $\{x_1,x_2,y,t_1,t_2\}$ of $G$. In the structure of semidirect product $\F{3} \rtimes_\sigma \F{2}$ we will term {\em horizontal subgroup} the normal subgroup $\F{3} = \langle x_1,x_2,y \rangle$ and {\em vertical subgroup} the subgroup $\F{2} = \langle t_1,t_2 \rangle$. Any element is uniquely written as a concatenation $t w$ where $t$ is a {\em vertical element}, i.e.~an element in the vertical subgroup, and $w$ is a {\em horizontal element}, i.e.~an element in the horizontal subgroup. We denote by $\mathcal A$ the alphabet
over $S \cup S^{-1}$ and by $\pi$ the map which, to a given word in $\mathcal A$, assigns
the unique element of $G$ that it defines. A {\em reduced word} is a word without any cancellation $x x^{-1}$ or $x^{-1} x$.  Words consisting of vertical (resp.~horizontal) letters are {\em vertical} (resp.~{\em horizontal}) {\em words}.
A {\em reduced representative} of an element $g$ in $G$ is a reduced word in the alphabet $\mathcal A$ whose image under $\pi$ is $g$.

We denote by $\Gamma$ the Cayley graph of $G$ with respect to $S$. Since the vertices of $\Gamma$ are in bijection with the elements of $G$, we do not distinguish between a vertex of $\Gamma$ and the element of $G$ associated to this vertex. The edges of $\Gamma$ are oriented: an edge of $\Gamma$ is denoted by the pair ``(initial vertex of the edge, terminal vertex of the edge)''. The edges are labeled with the elements in $S \cup S^{-1}$. For instance the edge $(g,gx_i)$ has label $x_i$, whereas the edge $(gx_i,g)$ has label $x^{-1}_i$. If $x$ is the label of an edge we will term this edge {\em $x$-edge}. If $E$ is an oriented edge then $E^{-1}$ is the same edge with the opposite orientation. For instance
$(g,gt_i)^{-1} = (gt_i,g)$. A {\em reduced edge-path} in $\Gamma$ is an edge-path which reads a reduced word. When considering $\Gamma$ as a cellular complex, there is exactly one $1$-cell associated to the two edges $(g,gs)$ ($s \in S$) and $(gs,g)$, and each orientation of this $1$-cell corresponds to one of these edges.

\begin{lemma} \hfill
\label{yves}

With the notations above, the group $G$ admits $S_{\mathrm{min}} :=  \{y,t_1,t_2\} \subset S$ as a generating set.
\end{lemma}

\begin{proof}
For $i \in \{1,2\}$ we have $x_i = y^{-1} t^{-1}_i y t_i$ hence the lemma.
\end{proof}

As a straightforward consequence:

\begin{corollary} \hfill

Let $\chi_i$ be the set of $1$-cells of $\Gamma$ associated to edges with label $x^{\pm 1}_i$, and let $\Gamma_c$ be the closure of the complement of $\chi_1 \cup \chi_2$ in $\Gamma$. Then $\Gamma_c$ is ($G$-equivariantly homeomorphic to) the Cayley graph of $G$ with respect to the generating set $S_{\mathrm{min}}$ defined in Lemma \ref{yves}.
\end{corollary}

\subsection{Space with walls structure}

{\em Spaces with walls} were introduced in \cite{HaglundPaulin} in order to check the Haagerup property. A space with walls is a pair $(X,\mathcal W)$ where $X$ is a set and $\mathcal W$ is a family of partitions of $X$ into two classes, called {\em walls}, such that for any two distinct points $x, y$ in $X$ the number of walls $\omega(x,y)$ is finite. This is the {\em wall distance} between $x$ and $y$. We say that a discrete groups {\em acts properly} on a space with walls $(X,\mathcal W)$ if it leaves invariant $\mathcal W$ and for some (and hence any) $x \in X$ the function $g \mapsto \omega(x,gx)$ is proper on $G$.

\begin{theorem}[\cite{HaglundPaulin}] \hfill

 A discrete group $G$ which acts properly on a space with walls satisfies the Haagerup property.
 \end{theorem}

 In order to get Theorem \ref{the theorem} we will need the following (stronger) result (we refer to \cite{Nica} for a similar statement):

\begin{theorem}[\cite{ChatterjiNiblo}] \hfill
\label{ceci est un rappel}

Let $G$ be a discrete group which acts properly on a space with walls $(X,\mathcal W)$. Say that two walls $(u,u^c) \in \mathcal W$ and $(v,v^c) \in \mathcal W$ {\em cross}
 if all four intersections $u \cap v$, $u \cap v^c$, $u^c \cap v$ and $u^c \cap v^c$
 are non-empty. Let $I({\mathcal W})$ be the (possibly infinite) supremum of the cardinalities
of finite collections of walls which pairwise cross. Then $G$ acts properly isometrically on some $I({\mathcal W})$-dimensional CAT(0) cube complex. In particular it satisfies the Haagerup property.
\end{theorem}

\section{Horizontal and vertical walls}

\subsection{Definition and stabilizers}

\begin{definition} \hfill
\label{ensemble elementaire}

The {\em horizontal block ${\mathcal Y}$} is the set of all the elements in $G$ which admit $t y w$, with $t$ a vertical word and $w$ a horizontal word, as a reduced representative.
 A {\em horizontal wall} is a left-translate $g({\mathcal Y},{\mathcal Y}^c)$, $g \in G$.

The {\em vertical $j$-block ${\mathcal V}_j$} is the set of all the elements in $G$ which admit $t_j t w$, with $t$ a vertical word and $w$ a horizontal word, as a reduced representative.
A {\em vertical $j$-wall} is a left-translate  $g({\mathcal V}_j,{\mathcal V}^c_j)$, $g \in G$.
\end{definition}

\begin{figure}[htbp]
{\centerline{\includegraphics[height=8cm, viewport = 95 530 460 820,clip]{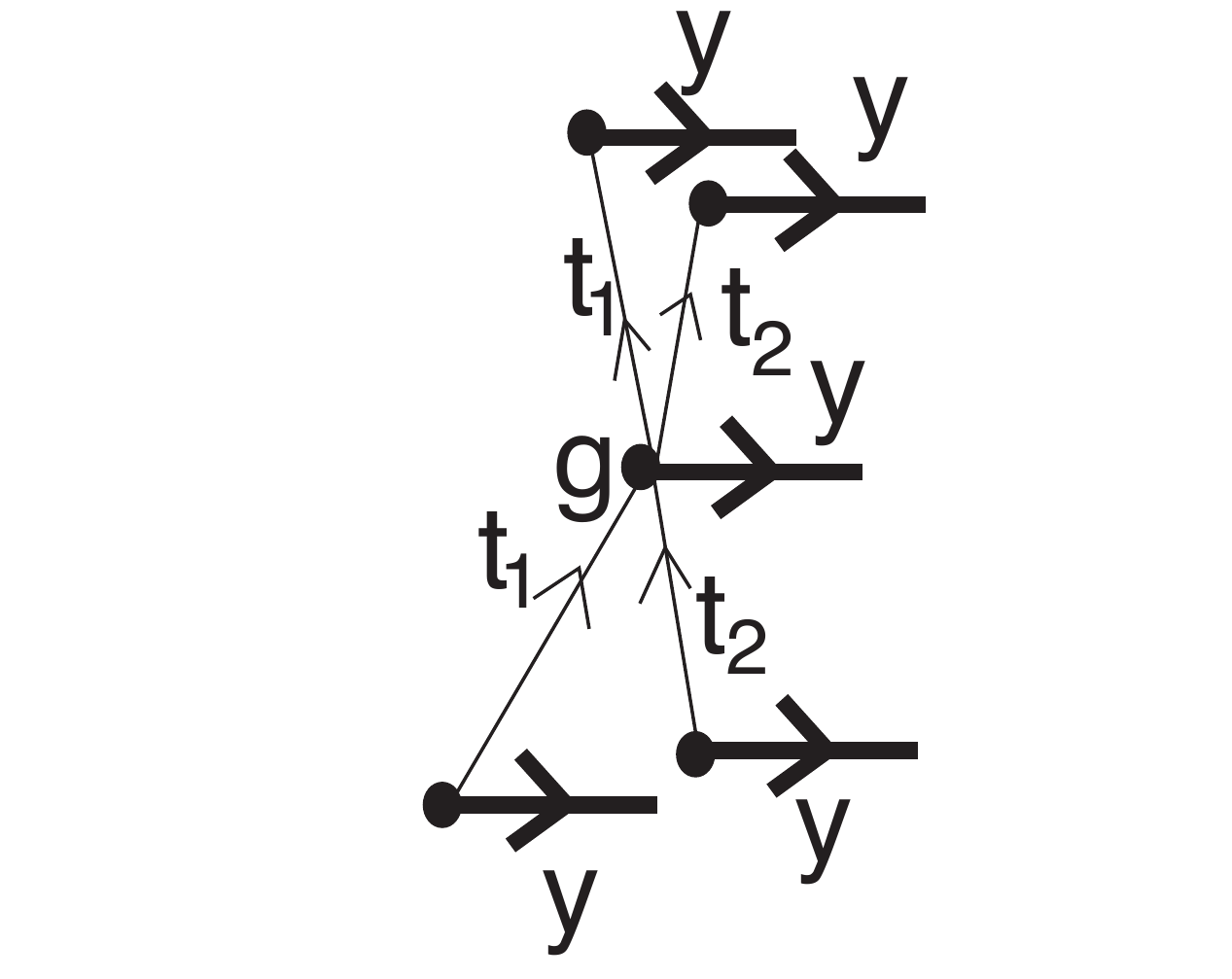}}} \caption{\label{dessin} Horizontal wall}
\end{figure}

See Figure \ref{dessin} for an illustration of the horizontal walls. By definition of a reduced representative, in the definition of ${\mathcal Y}$ (resp.~of ${\mathcal V}_j$), $w$ (resp.~$t$) does not begin with $y^{-1}$ (resp.~with $t^{-1}_j$).

\begin{lemma} \hfill
\label{stabilisateurs1}

The collection of all the horizontal walls is $G$-invariant for the left-action of $G$ on itself.  The same assertion is true for the collection of all the vertical walls. Moreover:

\begin{enumerate}
  \item The left $G$-stabilizer of any horizontal wall is a conjugate of the vertical subgroup.

  \item The horizontal subgroup is both the left and right $G$-stabilizer of any vertical wall.

\end{enumerate}
\end{lemma}

\begin{proof}
By definition the collection of either all the horizontal or all the vertical walls consists of all the left $G$-translates of the horizontal or vertical walls $({\mathcal Y},{\mathcal Y}^c)$ or $({\mathcal V}_j,{\mathcal V}^c_j)$ so that it is invariant under the left $G$-action.

Let $g \in {\mathcal Y}$. Then $g=ty w$ for some $t$ in the vertical subgroup and $w$ in the horizontal one. If $t^\prime$ is another element in the vertical subgroup, $t^\prime g = t^\prime t y w \in {\mathcal X}_i$. Thus the vertical subgroup is in the $G$-stabilizer of ${\mathcal Y}$. By the relation $ut = t \sigma(t)(u)$ for $u$ in the horizontal subgroup (recall that $\sigma \colon \F{2} \hookrightarrow \Aut{\F{3}}$ is the monomorphism such that $G = \F{3} \rtimes_\sigma \F{2}$) we get $ug=t\sigma(t)(u) y w$ so that $\langle u \rangle$ does not stabilize ${\mathcal Y}$. Since any element of $G$ is the concatenation of a vertical element with a horizontal one, these observations imply that the $G$-stabilizer of ${\mathcal Y}$ is the vertical subgroup. Since the horizontal walls are left $G$-translates of the wall $({\mathcal Y},{\mathcal Y}^c)$, the $G$-stabilizer of a horizontal wall is a conjugate of the vertical subgroup.

The proof for the stabilizers of the vertical walls are similar and easier: just observe that since the horizontal subgroup is normal in $G$, it is useless to take its conjugates.
\end{proof}

\subsection{Finiteness of horizontal and vertical walls between any two elements}

\begin{proposition} \hfill
\label{vertical}

There are a finite number of vertical walls between any two elements.
\end{proposition}

\begin{proof}
The vertical walls are the usual walls used to prove that the free group $\F{2}$ satisfies the Haagerup property. Thus there are a finite number (in fact one) of vertical walls between $g \in G$ and $g t_i$ with $t_i$ a vertical generator. By Lemma \ref{stabilisateurs1} each vertical wall is stabilized by the right-action of the horizontal subgroup. Thus no vertical wall separates $g$ from $gs$, $g \in G$ and $s$ a horizontal generator $x_i$ or $y$. The proposition follows.
\end{proof}

\begin{proposition} \hfill
\label{horizontal}

There are a finite number of horizontal walls between any two elements in $G$.
\end{proposition}

This proposition is a little bit more difficult than the previous one and we need a preliminary lemma:

\begin{lemma} \hfill
\label{facile}

Each side of each horizontal wall is invariant under the right-action of the vertical subgroup: if $(\mathcal H,{\mathcal H}^c)$ is an horizontal
wall then for any element $t$ of the vertical subgroup we have ${\mathcal H} t = \mathcal H$ and ${\mathcal H}^c t = {\mathcal H}^c$.
In particular, each horizontal wall is invariant under the right-action of the vertical subgroup.
\end{lemma}

\begin{proof} We begin with the

\begin{claim}
Let $Y$ be the intersection of ${\mathcal Y}$ with the horizontal subgroup $\F{3}$. Then $Y$ is invariant under the conjugation-action of the vertical subgroup $\F{2}$.
\end{claim}

\begin{proof}
 Let $w \in Y$. We write the reduced word $w=ym_0 y^{\epsilon_1} m_1 \cdots y^{\epsilon_i} m_i \cdots y^{\epsilon_k}m_k$, $k \geq 0, \epsilon_i \in \{\pm 1\}$, with $m_i$ a reduced group word in the letters $x^{\pm 1}_1,x^{\pm 1}_2$ (since $w$ is reduced for each $i$ satisfying $\epsilon_i+\epsilon_{i+1}=0$ we have $m_i\neq 1$).
Then $t^{-1}_1 w t_1 = y \mu_0 y^{\epsilon_1} \mu_1 \cdots y^{\epsilon_i} \mu_i \cdots y^{\epsilon_k} \mu_k$, where $\mu_i$ has the form $a_i m_i b^{-1}_i$ with

\begin{itemize}
    \item $a_i=x_1$ if $\epsilon_i=1$ and $a_i=1$ otherwise,
    \item $b_i=x^{-1}_1$ if $\epsilon_{i+1}=-1$ and $b_i=1$ otherwise,
\end{itemize}

for $i=1,\cdots,k$ and setting $\epsilon_{k+1}=1$.

Therefore, if $i<k$ and $\epsilon_i+\epsilon_{i+1}=0$ then $\mu_i$ is conjugate to $m_i$ so that it is non trivial. Whenever $\epsilon_i + \epsilon_{i+1} \neq 0$ no cancellation might occur after reduction between $y^{\epsilon_i}$ and $y^{\epsilon_{i+1}}$ even if $\mu_i$ is reduced to the trivial word. Thus, after writing the $\mu_i$ as reduced words, the word $y\mu_0y^{\epsilon_1}\mu_1\dots y^{\epsilon_k}\mu_k$ we eventually get is reduced, so that $t^{-1}_1wt_1 Y \subset Y$. The reverse inclusion being analogous we have $t^{-1}_1wt_1 Y = Y$ and similarly $t^{-1}_2 w t_2 Y = Y$.
\end{proof}

Now $\displaystyle {\mathcal Y} =\bigcup_{t \in \F{2}} tY=\bigcup_{t \in \F{2}} (tYt^{-1})t$ which by the claim is equal to $\displaystyle \bigcup_{t \in \F{2}} Y t$. Therefore $u {\mathcal Y} t = u {\mathcal Y}$ for any $t$ in the vertical subgroup and for any $u$ in the horizontal one. Each one of the previous equalities holds when substituting ${\mathcal Y}^c$ for $\mathcal Y$ so that in particular $u({\mathcal Y},{\mathcal Y}^c)t = u ({\mathcal Y},{\mathcal Y}^c)$ for any $t$ in the vertical subgroup and for any $u$ in the horizontal one. Lemma \ref{facile} is proved.
\end{proof}

\begin{proof}[Proof of Proposition \ref{horizontal}]
 By Lemma \ref{facile}, there is no horizontal wall between any two elements $g$ and $gt_j$, $j=1,2$. On the other hand we have the following
 
 \begin{claim}
\label{portes ouvertes}
For any $g \in G$, $\displaystyle \bigcup_{t \in \F{2}} (gt,gty)^{\pm 1}$
disconnects $\Gamma_c$ in two connected components, which are the two sides of the horizontal wall $g(\mathcal Y,{\mathcal Y}^c)$. 
\end{claim}

Thus there is exactly one horizontal wall between any two elements $g$ and $gy$. By Lemma \ref{yves} we get the finiteness of the number of horizontal walls between any two elements in $G$.
\end{proof}

\section{Vertizontal walls}

Before beginning, we recall that $\Gamma_c$ denotes the Cayley graph of $G$ with respect to $S_{\mathrm{min}} = \{y,t_1,t_2\}$ (see Lemma \ref{yves}). We also recall that in a Cayley graph, $(g,gt_i)$ denotes the edge with label $t_i$ (a $t_i$-edge) oriented from $g$ to $gt_i$ and $(gt_i,g)$ the same edge with the opposite orientation (its label is thus $t^{-1}_i$, it is a $t^{-1}_i$-edge). Finally, if $E$ is an oriented edge then $E^{-1}$ denotes the same edge with the opposite orientation. In particular ${(g,gt_i)}^{-1} = (gt_i,g)$.

\subsection{Definition and stabilizers}

\begin{definition}
\label{definition enfin} \hfill

With the notations above: let $H_i := \langle x_{i+1},t_{i+1},y x_i y^{-1} t_i, x_i t^{-1}_i \rangle$ ($i=1,2 \mbox{ mod } 2$), let $E^+_i :=  H_i (e,t_i)$, $E^-_i := H_i (t_i,e)$ and $E_i := E^+_i \cup E^-_i$.

The {\em $i$-block} ${\mathcal T}_i$ is the set of all the elements in $G$ which are connected to the identity vertex $e$ by an edge-path in $\Gamma_c \setminus E_i$.
\end{definition}

See Figures \ref{elementaires} and \ref{tore}. Beware that Figure \ref{elementaires} might be slightly misleading when considering the edge-paths $y x_i y^{-1}$: they are indeed preserved under the right-action of $t_i$ but this is a consequence of the fact that a cancellation occurs between $\sigma(t_i)(x_i)=x_i$ and $\sigma(t_i)(y^{-1}) = x^{-1}_i y^{-1}$. Since we did not draw the images of the edges in these edge-paths, this cancellation does not appear in the figure.

\begin{figure}[htbp]
{\centerline{\includegraphics[height=13cm, viewport = 30 105 598 810,clip]{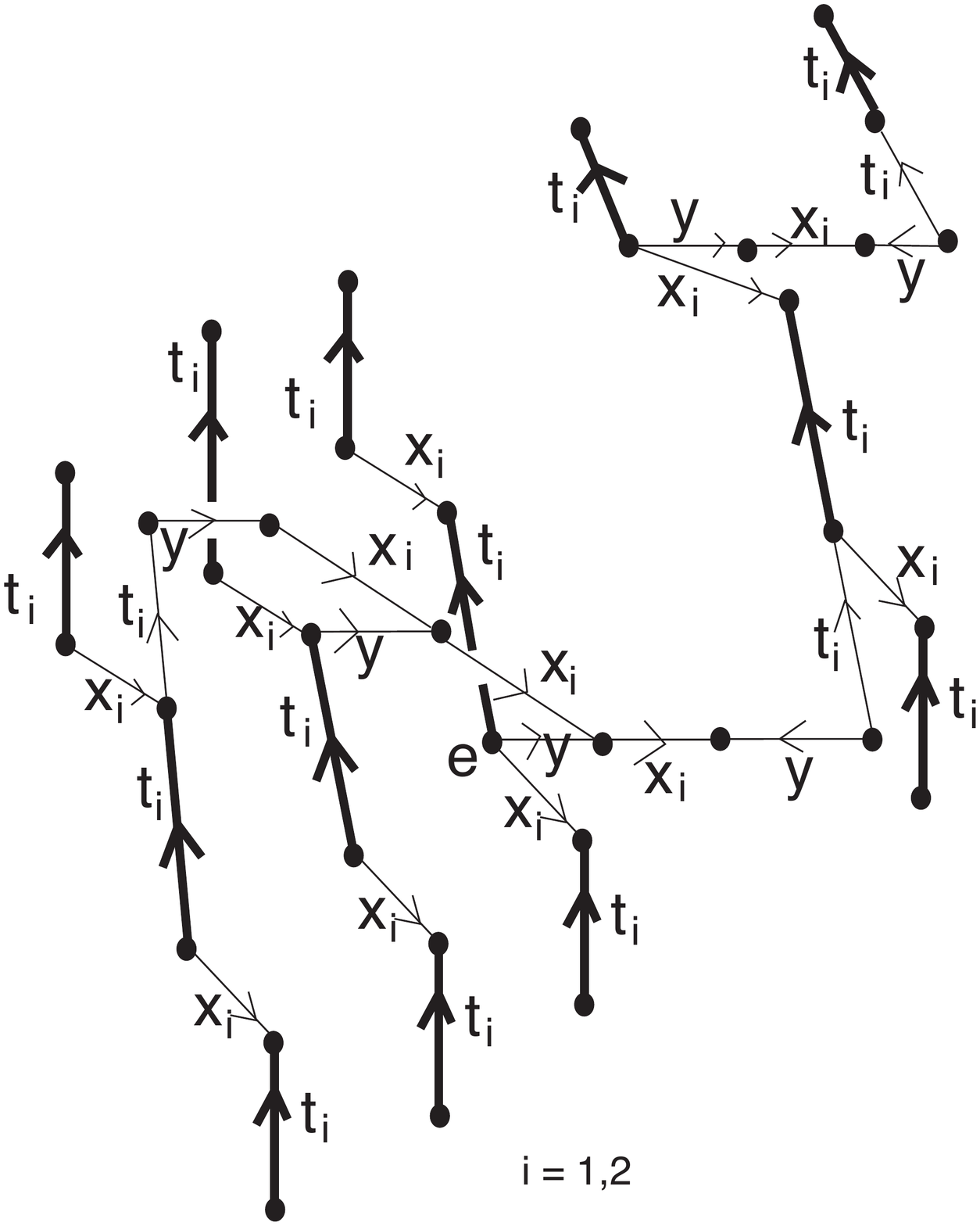}}} \caption{\label{elementaires} Some edges in $E_i$}
\end{figure}

\begin{figure}[htbp]
{\centerline{\includegraphics[height=7cm, viewport = 30 470 598 830,clip]{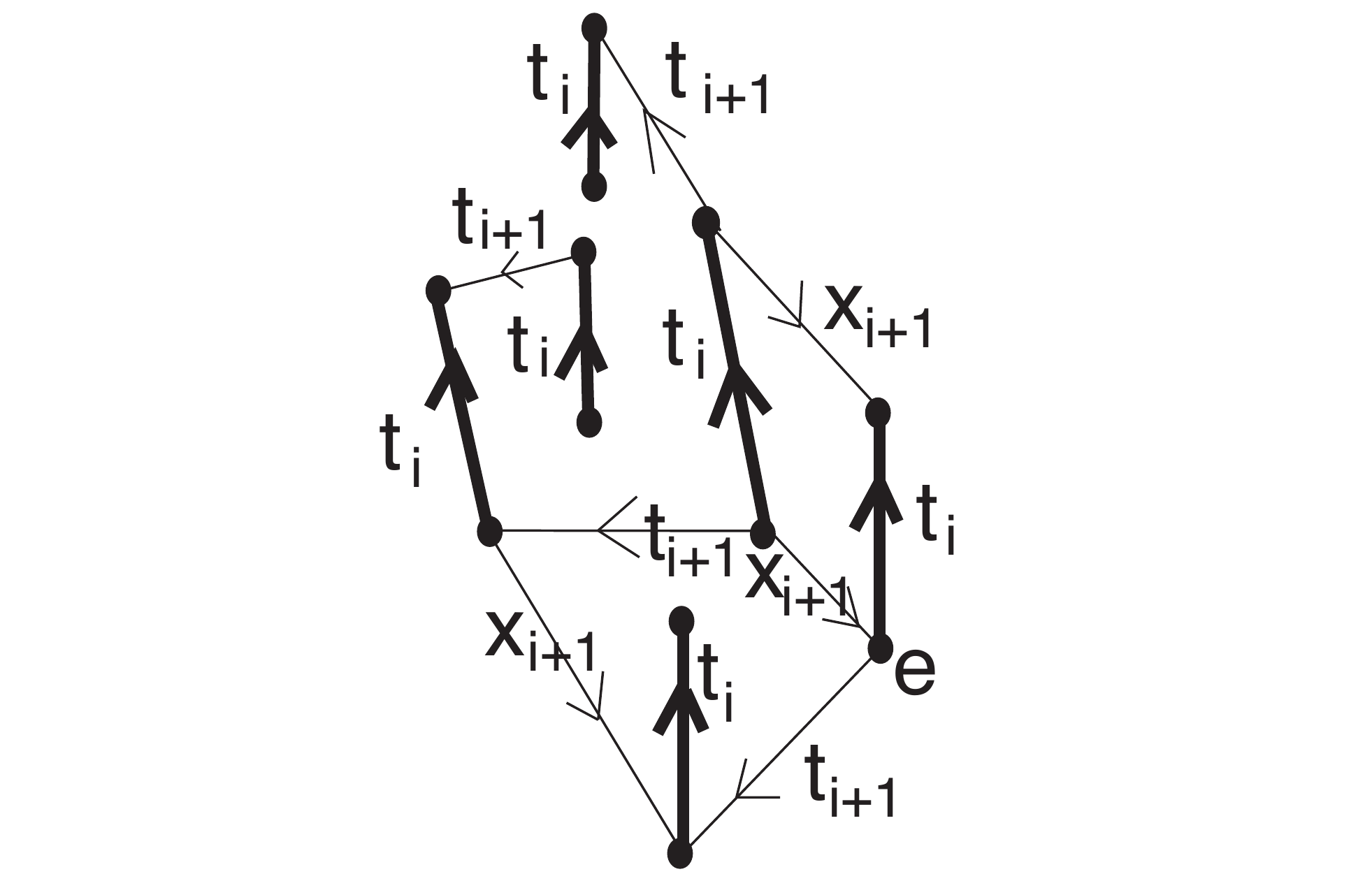}}} \caption{\label{tore} Other edges in $E_i$}
\end{figure}

\begin{remark} \hfill
\label{une remarque}

Observe that $x_{i+1} = y^{-1} t^{-1}_{i+1} y t_{i+1}$, $y x_i y^{-1} t_i = y t_i y^{-1}$ and $x_i t^{-1}_i = y^{-1} t^{-1}_i y$, i.e.~$H_i = \langle y^{-1} t^{-1}_{i+1} y t_{i+1},t_{i+1},y t_i y^{-1},y^{-1} t_i y \rangle$. In particular, for any $k \in \mz$, $y t^k_i y^{-1}$ and $y^{-1} t^k_i y$ are in $H_i$. See Figure \ref{je rajoute}.
\end{remark}

\begin{lemma}
\label{enfin 0} \hfill

All the initial (resp.~terminal) vertices of the edges in $E^+_i$ are connected to $e$ (resp.~to $t_i$) by an edge-path in $\Gamma_c \setminus E_i$.
\end{lemma}

\begin{proof}
The edge-path $(e,y^{-1}) (y^{-1},y^{-1} t^{-1}_i)  (y^{-1} t^{-1}_i,y^{-1} t^{-1}_i y)$ connects
 $e$ to $x_it^{-1}_i$ in $\Gamma_c \setminus E_i$. The edge-path $(e,y) (y,y t_i) (y t_i,y t_i y^{-1})$ connects $e$ to $t_i y x_i y^{-1} = y x_i y^{-1} t_i$ in $\Gamma_c \setminus E_i$. The edge $(e,t_{i+1})$ (indices are written modulo $2$) connects $e$ to $t_{i+1}$ in $\Gamma_c \setminus E_i$. The edge-path $(e,t^{-1}_{i+1}) (t^{-1}_{i+1},t^{-1}_{i+1} y) (t^{-1}_{i+1} y,t^{-1}_{i+1} y t_{i+1}) (t^{-1}_{i+1} y t_{i+1},t^{-1}_{i+1} y t_{i+1} y^{-1})$ connects $e$ to $x_{i+1}$.

Since $E_i$ is left $H_i$-invariant, taking and concatenating the left-translates of the previous edge-paths by elements of $H_i$ we connect all the initial vertices of the edges in $\langle x_{i+1}, t_{i+1}, y x_i y^{-1} t_i, x_i t^{-1}_i \rangle (e,t_i) = E_i$ by edge-paths in $\Gamma_c \setminus E_i$.

For connecting $t_i$ to $x_i$ (resp.~$t_i$ to $y x_i y^{-1} t^2_i$, $t_i$ to $x_{i+1} t_i$) in $\Gamma_c \setminus E_i$ just take the left-translate by $t_i$ of the edge-path between $e$ and $x_i t^{-1}_i$ (resp.~between $e$ and $t_i y x_i y^{-1}$, between $e$ and $x_{i+1}$): indeed recall that $t_i x_i = x_i t_i$, $t_i x_{i+1} = x_{i+1} t_i$ and $t_i y x_i y^{-1} = y x_i y^{-1} t_i$. For connecting $t_i$ to $t_{i+1} t_i$: first connect $t_i$ to $x_i$ by the edge-path given above, then $x_i$ to $x_i t_{i+1}$ by the edge $(x_i,x_it_{i+1})$, then $x_i t_{i+1}$ to $x_i t_{i+1} t_i x^{-1}_i = t_{i+1} t_i$ by the left $t_{i+1}$-translate of the edge-path between $t_i$ and $x_i$. We conclude as for the initial vertices.
\end{proof}

\begin{corollary}
\label{corollaire de enfin 0} \hfill

Any element in $G$ is connected either to $e$ or to $t_i$ by an edge-path in $\Gamma_c \setminus E_i$.
\end{corollary}

\begin{proof}
Let $g \in G$ and consider any edge-path $p$ in $\Gamma_c$ from (the vertex of $\Gamma_c$ associated to) $g$ to $e$. If $p \subset \Gamma_c \setminus E_i$ we are done. Otherwise $p$ passes through some edge in $E_i$, we denote by $q$ the subpath of $p$ from $g$ to the initial vertex $v$ of the first edge of $E_i$ in $p$. By Lemma \ref{enfin 0} the initial (resp.~terminal) vertex of each edge in $E^+_i$ is connected to $e$ (resp.~to $t_i$) in $\Gamma_c \setminus E_i$. Thus there is an edge-path $r$ in $\Gamma_c \setminus E_i$ from $v$ to either $e$ or $t_i$. The concatenation $qr$ gives an edge-path in $\Gamma_c \setminus E_i$ from $g$ to either $e$ or $t_i$.
\end{proof}

In what follows, in order to have a more readable text we will write each edge-path as a concatenation of the labels of its edges: in order to ensure that this defines an edge-path in $\Gamma_c \setminus E_i$ the reader will have each time to remind the starting-point of the edge-path.

\begin{lemma}
\label{enfin 1} \hfill

No edge-path in $\Gamma_c \setminus E_i$ containing only $y^{\pm 1}$-edges and $t^{\pm 1}_i$-edges connects $e$ to $t_i$.
\end{lemma}

Since the proof below is rather long, we first give an idea of what happens: the only way to go from $e$ to $x_i$ or $t_i$ without going through an $x^{\pm 1}_i$-edge (which does not exist in $\Gamma_c$) is to go through $t^{\pm 1}_i$-edges appearing in the relations $(y^{-1} t^{-1}_i y) t_i = x_i$ or $(y t_i y^{-1}) t_i (y t^{-1}_i y^{-1}) = t_i$ or their inverses (in parentheses the generators of $H_i$ involved - the $t^{\pm 1}_i$-edges in $E_i$ are those outside the parentheses). As the reader can check (see Figure \ref{je rajoute}), the edges in $E_i$ have been chosen to ``cut'' these relations. Each time one goes through a $y^{\pm 1}$-edge, one crosses a horizontal wall and since $e$ and $t_i$ are in the same side of any horizontal wall (Claim \ref{vraie claim 3}) one has to cross it back, and in particular when crossing $(g,gy)$ one has to go back to the right vertical orbit of $g$, i.e.~to some $gt$, $t \in \F{2}$ (Claim \ref{jpo}). If this is not a $H_i$-translate of the right vertical orbit of the identity of $G$, then one has a vertical short-cut from $g$ to $gt$. In this way, by an induction process, we prove that reduced edge-paths in $\Gamma_c \setminus E_i$ going from $e$ to $t_i$, which minimize the number of horizontal edges crossed, are concatenations of subpaths of the form $t^k_i y t^l_i y^{-1}$ (at least when assuming that one began with a $y$-edge as first horizontal edge - see Claim \ref{claim 4}). Claims \ref{claim 5} and \ref{claim 6} rely upon the fact that left-translating an edge-path in $\Gamma_c \setminus E_i$ by an element in $H_i$ yields an edge-path in $\Gamma_c \setminus E_i$: each subpath $y t^l_i y^{-1}$ defines such an element. Hence any subpath of $p$ remains in the same side of $E_i$ as $e$, hence the contradiction (Claim \ref{claim 7}). By Claim \ref{claim 8}, if the first horizontal edge in $p$ is a $y^{-1}$-edge, then there exists a reduced edge-path in $\Gamma_c \setminus E_i$ from $e$ to $t_i$ whose first horizontal edge is a $y$-edge and minimizing the number of horizontal edges it crosses, hence the conclusion (Claim \ref{claim 8} is proved by arguments similar to those exposed above).

\begin{figure}[htbp]
{\centerline{\includegraphics[height=9cm, viewport = 30 335 598 830,clip]{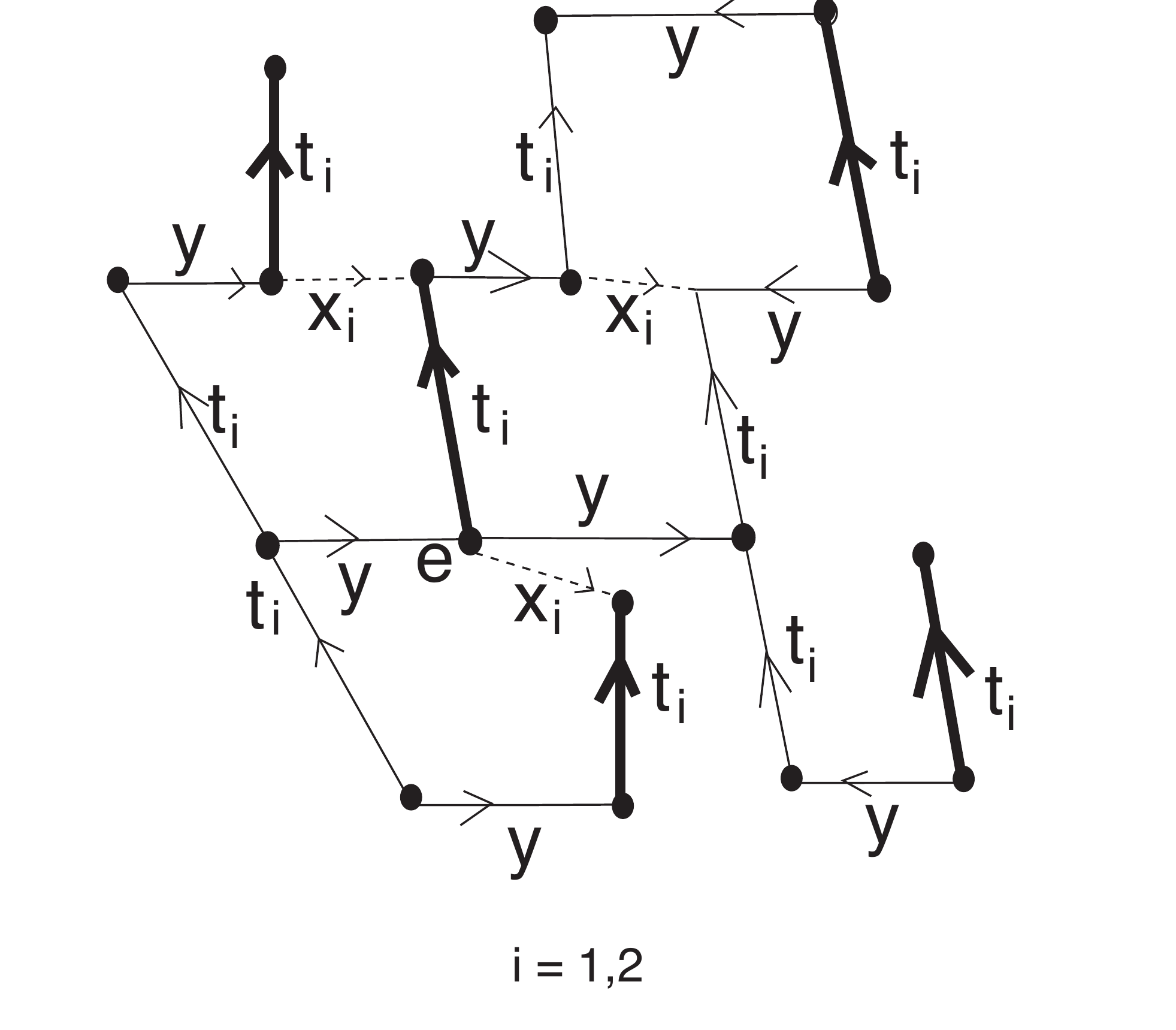}}} \caption{\label{je rajoute}}
\end{figure}

\begin{proof}
We begin with an easy assertion about horizontal walls:

\begin{claim}
\label{vraie claim 3}
The vertices $e$ and $t_i$ are in the same side of any horizontal wall $g({\mathcal Y},{\mathcal Y}^c)$. 
\end{claim}

\begin{proof}
By Lemma \ref{facile} each side of each horizontal wall is invariant under the right-action of the vertical subgroup. Of course $e$ and $t_i = e t_i$ belong to the same orbit for the right-action of the vertical subgroup hence the claim.
\end{proof}

Now an easy assertion about Definition \ref{definition enfin}:

\begin{claim}
\label{vraie claim 2}
Let $\phi \colon G \rightarrow \mz$ be the morphism defined by $\phi(y) = 1$ and $\phi(t_j) = 0$ for $j=1,2$ (since the sum of the exponents of the letters $y$ in the relators of $G$ is zero, this morphism is well-defined and since the elements $y, t_1, t_2$ generate $G$ - see Lemma \ref{yves} - it is defined over the whole group $G$). If $(g,gt_i)^{\pm 1}$ is any edge in $E_i$ then for any vertical element $t \in \F{2}$ the element $gt$ belongs to $\Ker{\phi}$.
\end{claim}

\begin{proof}
Definition \ref{definition enfin} gives $E_i =H_i (e,t_i)^{\pm 1}$ with $H_i =  \langle x_{i+1},t_{i+1},y x_i y^{-1} t_i, x_i t^{-1}_i \rangle$. Each element in the subgroup  $H_i$ belongs to $\Ker{\phi}$ since this is true for each generator. The claim follows from the fact that $e$ belongs to $\Ker{\phi}$ and all the elements in the vertical orbit of an element in $\Ker{\phi}$ also are in $\Ker{\phi}$ since $\phi(t_j) = 0$ for $j=1,2$.
\end{proof}

Let ${\mathcal P}_{e,t_i}$ be the set of all the reduced edge-paths in $\Gamma_c \setminus E_i$ from $e$ to $t_i$ passing only through $y^{\pm 1}$- and $t^{\pm 1}_i$-edges. Let ${\mathcal P}^{min}_{e,t_i} \subset {\mathcal P}_{e,t_i}$ be the subset formed by all the edge-paths in ${\mathcal P}_{e,t_i}$ minimizing the number of horizontal edges they cross. Let $p \in {\mathcal P}^{min}_{e,t_i}$.

\begin{claim}
\label{claim preliminaire}
There is at least one horizontal edge in $p$ and this $y$- or $y^{-1}$-edge begins at the vertical orbit of $e$.
\end{claim}

\begin{proof}
The set of all the vertical edges between vertices in a same vertical right-orbit forms a tree (a copy of the Cayley graph of $\F{2}$). Of course the vertices $e$ and $t_i$ belong to the same vertical right-orbit. The non-existence of a vertical edge-path between $e$ and $t_i$ then follows from the fact that $(e,t_i)^{\pm 1}$ belongs to $E_i$. The conclusion follows.
\end{proof}

\begin{claim}
\label{jpo}
Assume that $p$ goes through the $y$-edge $(g,gy)$, $g \in G$. Then there is $k \in \mz$ such that $p$ goes through the $y^{-1}$-edge $(gt^k_i y,gt^k_i)$.
\end{claim}

\begin{proof}
By definition $p$ goes from $e$ to $t_i$. By Claim \ref{vraie claim 3}, $e$ and $t_i$ are in the same side of any horizontal wall. By Claim \ref{portes ouvertes}, if $p$ goes through the $y$-edge $(g,gy)$ then it changes side in the horizontal wall $g (\mathcal Y,{\mathcal Y}^c)$ and has to cross back some $y$-edge with initial vertex in the orbit of $g$ under the right action of the vertical subgroup. This is exactly Claim \ref{jpo}.
\end{proof}

\begin{claim}
\label{claim 2}
Assume that the first horizontal edge in $p$ is a $y$-edge. Then $p$ admits an initial subpath of the form $t^{k_0}_i y t^{k_1}_i y^{-1}$ ($k_0$ might be zero).
\end{claim} 

\begin{proof}
Assume that $p$ does not satisfy the announced property. Then the second horizontal edge in $p$ also is a $y$-edge. Let $(\mathcal H,{\mathcal H}^c)$ be the horizontal wall associated to this $y$-edge.  By Claim \ref{vraie claim 3} $e$ and $t_i$ lie in the same side of $(\mathcal H,{\mathcal H}^c)$. Thus $p$ crosses back $(\mathcal H,{\mathcal H}^c)$. By Claim \ref{portes ouvertes} and since we assumed that $p$ does not cross any $t^{\pm 1}_{i+1}$-edge (indices are written modulo $2$), there is a non-trivial subpath $q_1$ of $p=q_0q_1q_2$ starting at the initial vertex $g$ of the above $y$-edge and going back to some $gt^m_i$ ($m \in \mz$). But the initial subpath of $p$ leading to $g$ is written as  $t^{k_0}_i y t^{k_1}_i$. The element in $G$ that it defines does not belong to $\Ker{\phi}$, see Claim \ref{vraie claim 2}. By this last claim no edge between two vertices in the right orbit of $g$ under the vertical subgroup belongs to $E_i$. Hence the vertical edge-path $r$ between $g$ and $gt^m_i$ is an edge-path in $\Gamma_c \setminus E_i$: it has the same endpoints as $q_1 \subset p=q_0q_1q_2$ and crosses at least one horizontal edge less than $q$ since $q_1$ is a reduced edge-path starting with a $y$-edge. Therefore the concatenation $q_0 r q_2$, possibly after reduction, defines an edge-path in ${\mathcal P}_{e,t_i}$ which crosses at least one horizontal edge less than $p$: this is a contradiction with $p \in {\mathcal P}^{min}_{e,t_i}$.
\end{proof}

\begin{claim}
\label{claim 3}
Assume that $p$ admits an initial subpath $q_0$ of the form $t^{k_0}_i y t^{k_1}_i y^{-1} \cdots t^{k_{2j}}_i y t^{k_{2j+1}}_i y^{-1}$. If the first horizontal edge following $q_0$ in $p$ is a $y$-edge then $p$ admits an initial subpath of the form $q_0 t^{k_{2j+2}}_i y t^{k_{2j+3}}_i y^{-1}$.
\end{claim}

\begin{proof}
This amounts to proving that the second horizontal edge following $q_0$ is a $y^{-1}$-edge. The argument for proving this claim is exactly the same as for Claim \ref{claim 2}: if the second horizontal edge following $q_0$ in $p$ also were a $y$-edge then the vertical edge-path from its initial vertex $g$ to $g t^m_i$ is in $\Gamma_c \setminus E_i$: indeed the edge-path in $p$ from $e$ to $g$ reads  $t^{k_0}_i y t^{k_1}_i y^{-1} \cdots t^{k_{2j}}_i y t^{k_{2j+1}}_i y^{-1} t^{k_{2j+2}} y$ and so is not in $\Ker{\phi}$ (see Claim \ref{vraie claim 2}). The conclusion is as in Claim \ref{claim 2}.
\end{proof}

\begin{claim}
\label{claim 4}
Assume that the first horizontal edge in $p$ is a $y$-edge. Then $p$ is the reduced concatenation of edge-paths reading words of the form $t^{k_{2j}}_i y t^{k_{2j+1}}_i y^{-1}$ in $\Gamma_c \setminus E_i$.
\end{claim}

\begin{proof}
By Claim \ref{claim 2}, $p$ admits a non-trivial initial subpath reading $t^{k_0}_i y t^{k_1}_i y^{-1}$. By Claim \ref{claim 3} it suffices to prove that the horizontal edge in $p$ following this initial subpath is not a $y^{-1}$-edge. Assume that it is, i.e.~$p = t^{k_0}_i y t^{k_1}_i y^{-1} t^{k_{2}}_i y^{-1} \cdots$. By Claim \ref{jpo} the edge-path $q_1 = t^{k_{2}}_i y^{-1} \cdots$ following $q_0 = t^{k_0}_i y t^{k_1}_i y^{-1} $ in $p$ has to go back to some $g t^n_i$, if $g$ is the terminal vertex of $q_0$ (hence the initial vertex of a $y$-edge). Since $p \in {\mathcal P}^{min}_{e,t_i}$ there is an edge in $E^+_i$ between $g$ and $g t^n_i$. Since $t^{k_m}_i y t^{k_{m+1}}_i y^{-1} = y x^{k_{m+1}}_i y^{-1} t^{k_{m+1}}_i t^{k_m}_i \in H_i t^{k_m}_i$ the edge-path $q_0$ reads an element of the form $h t^n_i$ with $h \in H_i$. By left-translation by $h^{-1}$ we pull-back $q_1$ to an edge-path starting at $e$ and ending at some $t^l_i$ with $l > 0$ since there is an edge in $E^+_i$ between the initial and terminal vertex. Since there are at least two horizontal edges in $q_0$ we so get an edge-path in ${\mathcal P}_{e,t_i}$ from $e$ to $t_i$ which has less horizontal edges than $p$. This contradicts $p \in {\mathcal P}^{min}_{e,t_i}$.
\end{proof}

\begin{claim}
\label{claim 5}
Assume that $p$  is the reduced concatenation of edge-paths reading words of the form $t^{k_{2j}}_i y t^{k_{2j+1}}_i y^{-1}$ in $\Gamma_c \setminus E_i$ for $j$ from $0$ to $l$. Then $\displaystyle \sum^l_{j=0} k_{2j} \leq 0$.
\end{claim}

\begin{proof}
We proceed by induction on $l$. For $l=0$ we have $p=t^{k_0}_i y t^{k_1}_i y^{-1}$. Since $p$ starts at $e$ and $(e,t_i) \in E_i$ we necessarily have $k_0 \leq 0$. Let us assume that the claim holds at $l$ and let us prove that then it holds at $l+1$. We observe that for any non-negative integer $j$ the element $y t^{k_{2j+1}}_i y^{-1}$ is in $H_i$. Thus a left-translate by (the inverse of) such an element of an edge-path $q$ is in $\Gamma_c \setminus E_i$ if and only if $q$ already was. We first left-translate the edge-path reading $t^{k_2}_i y \cdots$ and starting at $t^{k_0}_i y t^{k_1}_i y^{-1}$ by $(y t^{k_{1}}_i y^{-1})^{-1}$: we get an edge-path starting at $e$, reading
$t^{k_0}_i t^{k_2}_i y t^{k_3}_i y^{-1} \cdots$ and lying in $\Gamma_c \setminus E_i$ since $p$ is in $\Gamma_c \setminus E_i$. Since $(e,t_i) \in E_i$ this implies $k_0 + k_2 \leq 0$. We continue the process by left-translating by
$(y t^{k_3}_i y^{-1})^{-1}$ the subpath of $p$ starting with $t^{k_4}_i y$, and more generally by $(y t^{k_{2j+1}}_i y^{-1})^{-1}$ the subpath of $p$ starting with $t^{k_{2j+2}}_i y$. We eventually get $k_0 + k_2 + \cdots + k_{2l} \leq 0$ and the claim is proved. 
\end{proof}

\begin{claim}
\label{claim 6}
Let $\displaystyle g = \prod^{l}_{j=0} t^{k_{2j}}_i y t^{k_{2j+1}}_i y^{-1}$ be an element in $G$. Let $\phi_i \colon G \rightarrow \mz$ be the map which to an element $g$ assigns the sum of the exponents of the letters $x_i$ appearing in the unique reduced representative of $g$ of the form $wt$ where $w$ is a reduced horizontal word and $t$ is a reduced vertical one (beware that $\phi_i$ is not a morphism since its values on certain relators is non-zero). Then $\displaystyle \phi_i(g) = \sum^l_{j=0} k_{2j+1}$. 
\end{claim}

\begin{proof}
We prove by induction on $l$ that when writing $g \in G$ with the generating set $S$ we have $$g= y x^{k_1+k_3+\cdots+k_{2l+1}}_i y^{-1} t^{k_0+k_1+\cdots+k_{2l}+k_{2l+1}}_i.$$ If $l=0$ then $g = t^{k_0}_i y t^{k_1}_i y^{-1} = y x^{-k_0}_i t^{k_0+k_1}_i  y^{-1} =y x^{-k_0} x^{k_0+k_1}_i y^{-1} t^{k_0+k_1}_i= y x^{k_1}_i y^{-1} t^{k_0+k_1}_i$. So the assertion holds for $l=0$. Assume that it holds for $l$. Then if $\displaystyle g = \prod^{l+1}_{j=0} t^{k_{2j}}_i y t^{k_{2j+1}}_i y^{-1} = (\prod^{l}_{j=0} t^{k_{2j}}_i y t^{k_{2j+1}}_i y^{-1}) (t^{k_{2l+2}}_i y t^{2k_{2l+3}}_i y^{-1})$ by induction hypothesis we get $$g =  (y x^{k_1+k_3+\cdots+k_{2l+1}}_i y^{-1} t^{k_0+k_1+\cdots+k_{2l+1}}_i) (t^{k_{2l+2}}_i y t^{k_{2l+3}}_i y^{-1})$$ hence $$g= y x^{k_1+k_3+\cdots+k_{2l+1}+\cdots}_i y^{-1} y x^{-k_0-k_1-\cdots-k_{2l+2}}_i t^{k_0+k_1+\cdots+k_{2l+1}+k_{2l+2}+k_{2l+3}}_i y^{-1}$$ by permuting $t^{k_0+k_1+\cdots+k_{2l+1}+k_{2l+2}}_i$ with $y$ using the relation $t_iy= y x^{-1}_i t_i$ (notice that the exponent of $x_i$ is the opposite of the exponent of $t_i$). This is easier rewritten as $$g=y x^{-k_0-k_2-\cdots-k_{2l+2}}_i t^{k_0+k_1+\cdots+k_{2l+1}+k_{2l+2}+k_{2l+3}}_i y^{-1}$$ which gives $$g=y x^{-k_0-k_2-\cdots-k_{2l+2}}_i x^{k_0+k_1+\cdots+k_{2l+1}+k_{2l+2}+k_{2l+3}}_i y^{-1} t^{k_0+k_1+\cdots+k_{2l+1}+k_{2l+2}+k_{2l+3}}_i$$ by permuting $t^{k_0+k_1+\cdots+k_{2l+1}+k_{2l+2}+k_{2l+3}}_i$ with $y^{-1}$ using the relation $t_i y^{-1} = x_i y^{-1} t_i$ (notice that the exponent of $x_i$ is equal to the exponent of $t_i$). This is easier rewritten as $$g = y x^{k_1+k_3+\cdots+k_{2l+1}+k_{2l+3}}_i y^{-1} t^{k_0+k_1+\cdots+k_{2l+1}+k_{2l+2}+k_{2l+3}}_i$$ and the induction is complete. Since $\phi_i(g)$ is equal to the sum of the exponents of the $x_i$ in the previous writing, we get the claim.  
\end{proof}

\begin{claim}
\label{claim 7}
The first horizontal edge in $p$ is not a $y$-edge.
\end{claim}

\begin{proof}
We argue by contradiction and assume that the first horizontal edge in $p$ is a $y$-edge. By Claim \ref{claim 4}, $p$ is the reduced concatenation of edge-paths reading words of the form $t^{k_{2j}}_i y t^{k_{2j+1}}_i y^{-1}$ in $\Gamma_c \setminus E_i$ for $j$ from $0$ to $l$ ($l \geq 0$). By Claim \ref{claim 5}, $\displaystyle \sum^l_{j=0} k_{2j} \leq 0$. Since the element defined by $p$ is $t_i$ and the exponent of $t_i$ in $p$ is $\displaystyle \sum^l_{j=0} k_{2j} + \displaystyle \sum^l_{j=0} k_{2j+1}$, we have $\displaystyle \sum^l_{j=0} k_{2j} + \displaystyle \sum^l_{j=0} k_{2j+1} = 1$. Hence  $\displaystyle \sum^l_{j=0} k_{2j+1} > 0$. Claim \ref{claim 6} then gives $\phi_i(t_i) > 0$, which is an absurdity since $\phi_i(t_i) = 0$, hence the claim.
\end{proof}

\begin{claim}
\label{claim 8}
If there is $p \subset {\mathcal P}^{min}_{e,t_i}$ admitting a $y^{-1}$-edge as first horizontal edge, then there is $q \subset {\mathcal P}^{min}_{e,t_i}$ admitting a $y$-edge as first horizontal edge.
\end{claim}

\begin{proof}
The arguments are similar to those exposed above for proving that $p$ does not begin with a $y$-edge. Assume that the first horizontal edge in $p$ is a $y^{-1}$-edge, i.e.~$p = t^{k_0}_i y^{-1} \cdots$ with $k_0 \leq 0$. If $\phi$ is the morphism given in Claim \ref{vraie claim 2}, $\phi(t^{k_0} y^{-1}) = -1$ so that, by this same claim \ref{vraie claim 2}, there is no edge in $E_i$ between two vertices in the orbit of the terminal vertex of this $y^{-1}$-edge under the vertical subgroup. By Claim \ref{jpo}, $p$ has to cross back the associated horizontal wall. Moreover the number of horizontal edges in $p$ is minimal. Therefore this $y^{-1}$-edge is followed by an edge-path of the form $t^{k_1}_i y$ in $\Gamma_c \setminus E_i$, i.e.~$p = q_0 q_1$ with $q_0 = t^{k_0}_i y^{-1} t^{k_1}_i y$ and $q_1 \subset \Gamma_c \setminus E_i$. If the first horizontal edge in $q_1$ is also a $y^{-1}$-edge we repeat the argument.  Thus we eventually get a decomposition $p = t^{k_0}_i y^{-1} t^{k_1}_i y \cdots t^{k_{2m}}_i y^{-1} t^{k_{2m+1}}_i y p^\prime$ where:
 
 \begin{enumerate}
\item $p^\prime$ is non-trivial,
   \item the first horizontal edge in $p^\prime$ is a $y$-edge.
 \end{enumerate}
 
We noticed above that $k_0 \leq 0$. From Remark \ref{une remarque}, $y^{-1} t^{k_{2j+1}}_i y \in H_i$. Hence for any integer $j$ from $0$ to $m$ the left-translate of the subpath $t^{k_{2j+2}}_i y^{-1} \cdots$ by  $(y^{-1} t^{k_{2j+1}}_i y)^{-1}$ yields an edge-path in $\Gamma_c \setminus E_i$. We eventually get $\displaystyle \sum^m_{j=0} k_{2j} \leq 0$ (the same construction and argument have been exposed with more details in the proof of Claim \ref{claim 5}). In $G$ we have  $t^{k_0}_i y^{-1} t^{k_1}_i y \cdots t^{k_{2m}}_i y^{-1} t^{k_{2m+1}}_i y = x^{-k_1-k_3-\cdots-k_{2m+1}}_i  t^{k_0+k_1+\cdots+k_{2m}+k_{2m+1}}_i := g$ so that $\phi(g) \leq -\phi_i(x_i)$ ($\phi_i$ is the map from $G$  onto $\mz$ giving the exponent of $x_i$ - see Claim \ref{claim 6}). Hence the terminal vertex of any edge-path $t^{k_0}_i y^{-1} t^{k_1}_i y \cdots t^{k_{2m}}_i y^{-1} t^{k_{2m+1}}_i y$ starting at $e$ lies in the same side as $e$ in the grid $\langle x_i,t_i \rangle$.
 
  Thus there is $r \in \mn$ such that $t^{k_0}_i y^{-1} t^{k_1}_i y \cdots t^{k_m}_i y^{-1} t^{k_{m+1}}_i y t^r_i$ ends at some power of $x^{-1}_i t_i$, i.e.~as a group element defines a power $(x^{-1}_i t_i)^s$, $s \in \mz$. The left-translate of the edge-path $t^{k_0}_i y^{-1} t^{k_1}_i y \cdots t^{k_m}_i y^{-1} t^{k_{m+1}}_i y t^r_i (t^{-r}_i p^\prime)$ by $t^{-r}_i y^{-1} t^{-k_{m+1}}_i y t^{-k_{m}}_i y^{-1} \cdots y^{-1}t^{-k_1}yt^{-k_0}$ yields, after reduction, an edge-path in $\Gamma_c \setminus E_i$ from $e$ to $(x^{-1}_i t_i)^{-s} t_i$. Since $(x^{-1}_i t_i)^{-s} t_i = t_i (x^{-1}_i t_i)^{-s}$, by post-composing it with an edge-path reading $y^{-1} t^{-s}_i y$ if $s > 0$ and $y^{-1} t^s_i y$ if $s < 0$ we get an edge-path $q$ in $\Gamma_c \setminus E_i$ which belongs to ${\mathcal P}^{min}_{e,t_i}$ (it has at most the same number of horizontal edges as $p$), and the first horizontal edge of which is a $y$-edge since $p^\prime$ begins with a $y$-edge. 
  \end{proof}

If there exists an edge-path between $e$ and $t_i$ in $\Gamma_c \setminus E_i$ which goes only through horizontal and $t^{\pm 1}_i$-edges then there exists such an edge-path $p$ which is reduced and minimizes the number of horizontal edges that it crosses. By Claim \ref{claim preliminaire}, such an edge-path $p$ contains at least one horizontal edge. By Claim \ref{claim 7} the first horizontal edge in $p$ is not a $y$-edge. It follows by Claim \ref{claim 8}, that the first horizontal edge in $p$ neither is a $y^{-1}$-edge. We so eventually get that there exists no edge-path in $\Gamma_c \setminus E_i$ from $e$ to $t_i$ and Lemma \ref{enfin 1} is proved.
\end{proof}

\begin{lemma}  \hfill
\label{enfin 2}

If there exists an edge-path connecting $e$ to $t_i$ in $\Gamma_c \setminus E_i$ then there exists an edge-path composed only of horizontal edges and of $t^{\pm 1}_i$-edges connecting $e$ to $t_i$
in $\Gamma_c \setminus E_i$.
\end{lemma}

\begin{proof}
Assume the existence of an edge-path $p$ in $\Gamma_c \setminus E_i$ from $e$ to $t_i$. Then $p = w_0 w_1 \cdots w_{2k}$ where

\begin{enumerate}
  \item $w_{2j}$ is an edge-path passing only through horizontal and $t^{\pm 1}_i$-edges,
  \item $w_{2j-1}$ is an edge-path defining an element in the subgroup $\langle x_{i+1},t_{i+1} \rangle$ and does not pass through any $t^{\pm 1}_i$-edges.
\end{enumerate}

Since $\langle x_{i+1},t_{i+1} \rangle \subset H_i$, by a left-translation of $w_2 \cdots w_{2k}$ by $w^{-1}_1$ we get an edge-path $w^1_2 \cdots w^1_{2k}$ in $\Gamma_c \setminus E_i$ starting at the initial vertex of $w_1$ and ending at $w^{-1}_1 t_i$. Thus the concatenation $w_0 w^1_{2} \cdots w^1_{2k}$ defines an edge-path in $\Gamma_c \setminus E_i$ from $e$ to $w^{-1}_1 t_i$. By repeating this process we eventually get an edge-path $q = w_0 w^1_2 \cdots w^k_{2k}$ in $\Gamma_c \setminus E_i$ from $e$ to $w^{-1}_{2k-1} \cdots w^{-1}_1 t_i$ where $w^j_{2j}$ passes only through horizontal and $t^{\pm 1}_i$-edges whereas $w^{-1}_{2k-1} \cdots w^{-1}_1$ is an element in $\langle x_{i+1},t_{i+1} \rangle$. Since $q$ starts at $e$ and passes only through horizontal and $t^{\pm 1}_i$-edges, its terminal vertex is an element $g$ in $\langle y,t_i \rangle$. Let $h \in \langle x_{i+1},t_{i+1} \rangle$ with $g = h t_i$. Then $h = g t^{-1}_i$ so that $h \in \langle y,t_i \rangle$ since both $g$ and $t_i$ belong to $\langle y,t_i \rangle$. Therefore $h \in \langle x_{i+1},t_{i+1} \rangle \cap \langle y,t_i \rangle = \{e\}$. It follows that $q$ ends at $t_i$ so that $q$ is an edge-path as announced.
\end{proof}

\begin{corollary} 
\label{enfin} \hfill

There are exactly two connected components in $\Gamma_c \setminus E_i$: the connected component of $e$ and the connected component of $t_i$.
\end{corollary}

\begin{proof}
By Lemmas \ref{enfin 1} and \ref{enfin 2}, $e$ and $t_i$ lie in two distinct connected components of $\Gamma_c \setminus E_i$. By Corollary \ref{corollaire de enfin 0} these are the only two connected components of $\Gamma_c \setminus E_i$.
\end{proof}

By Corollary \ref{enfin}, if ${\mathcal T}_i$ denotes a $i$-block  (see Definition \ref{definition enfin}) then $({\mathcal T}_i,{\mathcal T}^c_i)$ is a wall so that the following definition makes sense:

\begin{definition} \hfill

A {\em vertizontal $i$-wall ($i=1,2$)} is any left-translate $g({\mathcal T}_i,{\mathcal T}^c_i)$, $g \in G$, of a $i$-block ${\mathcal T}_i$ (see Definition \ref{definition enfin}).
\end{definition}


\begin{lemma} \hfill
\label{stabilisateurs2}

The collection of all the vertizontal $i$-walls ($i=1,2$) is $G$-invariant for the left-action of $G$ on itself. The left $G$-stabilizer of any vertizontal $i$-wall ($i=1,2 \mbox{ mod } 2$) is a conjugate of the subgroup $H_{i} = \langle x_{i+1},t_{i+1},y x_i y^{-1} t_i,x_i t^{-1}_i \rangle$.
\end{lemma}

  \begin{proof}
The left $G$-invariance is obvious, as in the proof of Lemma \ref{stabilisateurs1}. Let us check the assertion about the left $G$-stabilizers. A vertizontal wall is a left $G$-translate of $({\mathcal T}_i,{\mathcal T}^c_i)$, where ${\mathcal T}_i$ is a $i$-block, see Definition \ref{definition enfin}. Thus its left $G$-stabilizer is conjugate in $G$ to the left $G$-stabilizer of $({\mathcal T}_i,{\mathcal T}^c_i)$. Since ${\mathcal T}_i$ is separated from ${\mathcal T}^c_i$ by the left $H_i$-translates of $(e,t_i)^{\pm 1}$ (see Definition \ref{definition enfin}), this left $G$-stabilizer is $H_i$.
\end{proof}

\subsection{Finiteness of the number of vertizontal walls between any two elements}

\begin{proposition} \hfill
\label{finitude}

There are a finite number of vertizontal walls between any two elements in $G$.
\end{proposition}

\begin{proof}
We consider the set of vertizontal $1$-walls (the proof is the same for the vertizontal $2$-walls). We work with the generating set $S_{\mathrm{min}} = \{y,t_1,t_2\}$ given by Lemma \ref{yves}. Since any $y^{\pm 1}$- and any $t^{\pm 1}_2$-edge lies in $\displaystyle \Gamma_c \setminus (\bigcup_{g \in G} gE_1)$ (see Definition \ref{definition enfin}) no vertizontal $1$-wall is intersected when passing from $g$ to $gy$ nor from $g$ to $gt_2$ whatever $g \in G$ is considered. Thus one only has to check which vertizontal $1$-walls are intersected when passing from $e$ to $t_1$. There is of course the wall $({\mathcal T}_1,{\mathcal T}^c_1)$. Assume that there is another wall $g ({\mathcal T}_1,{\mathcal T}^c_1)$. Then, by definition, this wall corresponds to the partition of $\Gamma_c$ in two components given by $g E_1$. Thus $(e,t_1)^{\pm 1} \in g E_1$. Let $a \in E_1$ with $(e,t_1)^{\pm 1} = ga$. By definition of $E_1$ there is $h \in H_1$ (see Definition \ref{definition enfin}) with $a = h(e,t_1)^{\pm 1}$ hence $(e,t_1)^{\pm 1} = gh (e,t_1)^{\pm 1}$. Since the stabilizer of any $1$-cell is trivial we get $g=h^{-1}$ so that $g E_1 = E_1$. This implies $g ({\mathcal T}_1,{\mathcal T}^c_1) = ({\mathcal T}_1,{\mathcal T}^c_1)$ and we are done.
\end{proof}

\section{A proper action}

However obvious, the following proposition is indispensable: 

\begin{proposition} \hfill
\label{action}

The set of all the horizontal, vertical and vertizontal walls defines a space with walls structure $(G,\mathcal W)$ for $G$. The left action of $G$ on itself  defines an action on this space with walls structure.
\end{proposition}

\begin{proof}
By Propositions \ref{vertical}, \ref{horizontal} and \ref{finitude} there are a finite number of walls between any two elements so that $(G,\mathcal W)$ is a space with walls structure.
\end{proof}

We now prove the following

\begin{proposition} \hfill
\label{proprete}

The action of $G$ on the space with walls structure $(G,\mathcal W)$ given by Proposition \ref{action} is proper.
\end{proposition}

\begin{proof}
Before beginning let us recall that what is important is the algebraic intersection-number of the edge-paths with each wall: if a given path $p$ intersects two times a wall $(W,W^c)$ first passing from $W$ to $W^c$ then crossing back from $W^c$ to $W$, this intersection-number is zero.

We work with the classical generating set $S=\{x_1,x_2,y,t_1,t_2\}$ of $G$. Each element $g \in G$ admits an unique reduced representative of the form $w t$ with $w$ a reduced horizontal word and $t$ a reduced vertical word. Since the vertical walls are the classical walls in the free group $\F{2}$, the number of vertical walls intersected goes to infinity with the number of letters in $t$. Thus we can assume that $g$ admits the reduced horizontal word $w$ as a reduced representative.

 Recall that the intersections of the horizontal walls with the horizontal subgroup give classical walls of the free group. By Lemma \ref{facile} horizontal walls are invariant under the right-action of the vertical subgroup. In particular any two $y^{\pm 1}$-edge in the reduced horizontal word $w$ define distinct horizontal walls. It follows that the number of horizontal walls intersected goes to infinity with the number of $y^{\pm 1}$-letters in $w$. Thus we can assume that $w$ contains only $x^{\pm 1}_i$-letters, $i=1,2$ i.e. we can assume that $w = x^{k_1}_{i_1} x^{k_2}_{i_2} \cdots x^{k_r}_{i_r}$ with $k_j \in \mz$ and $i_j \in \{1,2\}$, $i_j \neq i_{j+1}$. 
  
 Two vertizontal $i$-walls separate $e$ from $x_i$ in $\Gamma_c$: the $i$-wall $({\mathcal T}_i,{\mathcal T}^c_i)$ and the $i$-wall $y^{-1} ({\mathcal T}_i,{\mathcal T}^c_i)$. These are indeed the two walls intersected exactly once by the edge-path starting at $e$ and reading $y^{-1} t^{-1}_i y t_i$. Of course they also separate $e$ from $x^k_i$ ($k \in \mz$). The left-translates by $h \in G$ of these two $i$-walls separate $h$ from $hx^k_i$. This readily implies that there are at least two vertizontal walls intersected by any edge-path $x^{k_j}_{i_j}$ in $w$. Moreover the two $i$-walls given previously for passing from $e$ to $x^k_i$ are necessarily distinct from those given for passing from $x^k_i x^l_j$ ($j \neq i$) to $x^k_i x^l_j x^m_i$ ($k, l, m \in \mz$): indeed $x^k_i x^l_j$ ($i \neq j$) does not belong to the stabilizer of a vertizontal $i$-wall.  We so found a collection of $i$-walls intersected by the $x^{k_j}_{i_j}$ in $w$ which are all distinct and whose number goes to infinity with the number of times the letters $x^{\pm 1}_1$ and $x^{\pm 1}_2$ alternate in $w$ (since the number of intersections is increased by $2$ each times one reads a new word of the form $x^k_i$, $i=1$ or $i=2$). Therefore we can assume $w=x^k_1$ with $k \in \mn$.

The $2k$ left-translates by $x_1,x^2_1,\cdots,x^{k-1}_1$ of the vertizontal $1$-walls $({\mathcal T}_1,{\mathcal T}^c_1)$ and $y^{-1} ({\mathcal T}_1,{\mathcal T}^c_1)$ separate $e$ from $x^k_1$: these are indeed the walls crossed exactly once by the edge-path in $\Gamma_c$ starting at $e$, ending at $x^k_1$ and reading $y^{-1} t^{-k}_1 y t^k_1$. We so get the proposition.
\end{proof}

\section{The Haagerup property and dimension of the cube complex}

We give here, as corollaries of the construction developed above, the two main results we were interested in: the Haagerup property for $G$ and the, stronger, fact that $G$ acts properly isometrically on a cube complex (Theorem \ref{the theorem}).

\begin{corollary} \hfill

The group $G$ satisfies the Haagerup property.
\end{corollary}

\begin{proof}
By Propositions \ref{action} and \ref{proprete},  $G$ acts properly on a space with walls structure $(G,{\mathcal W})$.
By \cite{HaglundPaulin} $G$ satisfies the Haagerup property.
\end{proof}

\begin{corollary} \hfill
\label{allez}

The group $G$ acts properly isometrically on some $6$-dimensional cube complex, where $6$ is the supremum of the cardinalities of collections of walls which pairwise cross in the space with walls structure for $G$ given by Proposition \ref{action}.
\end{corollary}

\begin{proof}
Let $(G,\mathcal W)$ be the space with walls structure for $G$ given by Proposition \ref{action}. Let us recall that $\mathcal W$ is the set of all the horizontal, vertical and vertizontal walls. By Proposition \ref{proprete}, $G$ acts properly on $(G,{\mathcal W})$. By \cite{ChatterjiNiblo} $G$ acts properly isometrically on some $I({\mathcal W})$-dimensional cube complex where $I({\mathcal W})$ is the supremum of the cardinalities
of collections of walls which pairwise cross (see Theorem \ref{ceci est un rappel}).  


\begin{lemma} \hfill
\label{pour le rapporteur}

With the notations above: let $\mathcal F$ be a collection of walls in $(G,\mathcal W)$ which pairwise cross. Then there is at most one vertical wall and one horizontal wall in $\mathcal F$.
\end{lemma}

\begin{proof}
Vertical walls are the classical walls of the free group: the two sides of such a wall are separated by an edge of the Cayley graph with respect to a basis of the free group (a tree).  Thus two distinct such walls satisfy that one of the two sides of a wall properly contains a side of the other. Consequently, two distinct vertical walls do not pairwise cross. Let us now consider two distinct horizontal walls. By Claim \ref{portes ouvertes}, the two sides of a horizontal wall are separated by $\displaystyle \bigcup_{t \in \F{2}} (gt,gty)^{\pm 1}$. Thus, as it is the case for the free group, one of the two sides of a wall properly contains a side of the other. Lemma \ref{pour le rapporteur} is proved.
\end{proof}

\begin{lemma} \hfill
\label{pour le rapporteur II}

With the notations of Lemma \ref{pour le rapporteur}, for each $i \in \{1,2\}$:

\begin{enumerate}
  \item The vertizontal $i$-walls $({\mathcal T}_i,{\mathcal T}^c_i)$ and $y({\mathcal T}_i,{\mathcal T}^c_i)$ cross.
  \item There are at most two vertizontal $i$-wall in $\mathcal F$.
\end{enumerate}
\end{lemma}

\begin{proof}
The following claim is obvious:

\begin{claim}
\label{49}
For any $g \in G$ either $g E_i \cap E_i = \emptyset$ which is equivalent to $g \notin H_i$ or $g E_i = E_i$ which is equivalent to $g \in H_i$.
\end{claim}

Claim \ref{50} below is extracted from the proof of Lemma \ref{enfin 0}.

\begin{claim}
\label{50}
Let $g_0, g_1$ be the two initial (resp.~terminal) vertices of an edge in $g E^+_i$, $g \notin H_i$ (we recall that $E_i = E^+_i \cup E^-_i$ with $E^+_i = H_i (e,t_i)$). Then there is a reduced edge-path in $\Gamma_c \setminus gE_i$ between $g_0$ and $g_1$ satisfying the following properties:

\begin{itemize}
  \item It is a concatenation of edge-paths of four kinds: edge-path reading words of the form $(y t_i y^{-1})^{\pm 1}$, edge-paths reading words of the form $(y^{-1} t_i y)^{\pm 1}$ and edge-paths reading words of the form $t^{\pm 1}_{i+1}$ or ${(y^{-1} t^{-1}_{i+1} y t_{i+1})}^{\pm 1}$ ($i=1,2 \mbox{ mod } 2$).
  \item Both the initial and terminal vertices of each of the above subpaths are the initial (resp.~terminal) vertices of $t_i$-edges in $g E^+_i$.
\end{itemize}
\end{claim}

Assume that two distinct vertizontal $i$-walls $g_1 ({\mathcal T}_i,{\mathcal T}^c_i)$ and $g_2 ({\mathcal T}_i,{\mathcal T}^c_i)$ cross. Then (just apply a left-translation by $g^{-1}_1$) $({\mathcal T}_i,{\mathcal T}^c_i)$ and $g ({\mathcal T}_i,{\mathcal T}^c_i)$ cross, with $g = g^{-1}_1 g_2$. It is thus sufficient to prove that there is at most one left-coset $gH_i$ ($g \notin H_i$) such that $({\mathcal T}_i,{\mathcal T}^c_i)$
and $g  ({\mathcal T}_i,{\mathcal T}^c_i)$ cross.

There is a reduced edge-path $p$ in $\Gamma_c \setminus (E_i \cup gE_i)$ between $e$ and the initial vertex $g_0$ of some edge in $gE_i$. Without loss of generality assume $g_0 \in g {\mathcal T}_i$, which is equivalent to $(g_0,g_0 t_i) \in g E^+_i$. Then $\{e,g_0\} \subset {\mathcal T}_i \cap g {\mathcal T}_i$ so that in particular ${\mathcal T}_i \cap g {\mathcal T}_i \neq \emptyset$.

By Claim \ref{49}, since $(e,t_i) \in E_i$ (resp.~$(g_0,g_0 t_i) \in gE_i$ and $g \notin H_i$), we have $(e,t_i) \notin gE_i$ (resp.~$(g_0,g_0t_i) \notin E_i$). Therefore, setting $q = p (g_0,g_0t_i)$ we get a reduced edge-path $q \subset \Gamma_c \setminus E_i$ between $e$ and $g_0 t_i$ so that $g_0 t_i \in {\mathcal T}_i$. Moreover, since $(g_0,g_0 t_i) \in g E^+_i$, $g_0 t_i \in g{\mathcal T}^c_i$. Hence $g_0 t_i \in   {\mathcal T}_i \cap g{\mathcal T}^c_i$ so that ${\mathcal T}_i \cap g{\mathcal T}^c_i \neq \emptyset$. Similarly, setting $r = p^{-1} (e,t_i)$ we get an edge-path in $\Gamma_c \setminus gE_i$ so that $t_i \in g{\mathcal T}_i$. Since $t_i \in {\mathcal T}^c_i$ this implies $g {\mathcal T}_i \cap {\mathcal T}^c_i \neq \emptyset$.

At this point we so proved that ${\mathcal T}_i \cap g {\mathcal T}_i$, ${\mathcal T}_i \cap g{\mathcal T}^c_i \neq \emptyset$ and $g {\mathcal T}_i \cap {\mathcal T}^c_i \neq \emptyset$ (of course, if we had assumed $g_0 \in g {\mathcal T}^c_i$ instead of $g_0 \in g {\mathcal T}_i$ we would also have found three non-empty intersections among the four possible intersections between the different sides of the walls; however they would not have been the same but ${\mathcal T}_i \cap g {\mathcal T}^c_i$, ${\mathcal T}^c_i \cap g {\mathcal T}^c_i$ and ${\mathcal T}_i \cap g {\mathcal T}_i$). 

Assume now ${\mathcal T}^c_i \cap g{\mathcal T}^c_i \neq \emptyset$. Since $t_i \in {\mathcal T}_i^c$ there exists a reduced edge-path $s$ in $\Gamma_c \setminus E_i$ from $t_i$ to some element in $g {\mathcal T}^c_i$. Since $t_i \in g {\mathcal T}_i$ this edge-path $s$ crosses an edge $(g_1,g_1t_i)$ in $g E^+_i$, and we can assume that it crosses only one such edge. Let us denote by $q^\prime$ the subpath of $s$  from $t_i$ to $g_1$: $q^\prime$ is a reduced edge-path in $\Gamma_c \setminus (E_i \cup gE_i)$. Let us consider a reduced edge-path $q^{\prime \prime}$ in $\Gamma_c \setminus gE_i$ from $g_1$ to $g_0$ as given by Claim \ref{50}. Assume that $q^{\prime \prime}$ is contained in $\Gamma_c \setminus E_i$. Then $q^\prime q^{\prime \prime} p$ is an edge-path in $\Gamma_c \setminus E_i$ from $e$ to $t_i$ which is impossible. Therefore $q^{\prime \prime}$ crosses an edge in $E_i$. But, by construction (see Claim \ref{50}), the only $t^{\pm 1}_i$-edges crossed by $q^{\prime \prime}$ belong to subpaths of the form $(y t_i y^{-1})^{\pm 1}$ or $(y^{-1} t_i y)^{\pm 1}$ and the initial and terminal vertices of these subpaths are in $gE_i$. Thus these $t^{\pm 1}_i$-edges crossed by $q^{\prime \prime}$ are $t^{\pm 1}_i$-edges in $y g E_i$ or in $y^{-1} g E_i$. Since they belong to $ygE_i \cap E_i$ or to $ygE_i \cap E_i$, by Claim \ref{49} we get $yg \in H_i$ or $y^{-1}g \in H_i$. Hence $g \in yH_i$ or $g \in y^{-1}H_i$. From all which precedes, $y ({\mathcal T}_i,{\mathcal T}^c_i)$ and $y^{-1} ({\mathcal T}_i,{\mathcal T}^c_i)$ do not cross since, by a left-translation by $y$, if they would cross so would $y^2 ({\mathcal T}_i,{\mathcal T}^c_i)$ and $({\mathcal T}_i,{\mathcal T}^c_i)$ which has been proved to be false. We so got that if two vertizontal $i$-walls $\mathcal Z$ and ${\mathcal Z}^\prime$ cross then there is $g \in G$ such that, up to a permutation of $\mathcal Z$ and ${\mathcal Z}^\prime$, either $\mathcal Z = g({\mathcal T}_i,{\mathcal T}^c_i)$ and ${\mathcal Z}^\prime = gy({\mathcal T}_i,{\mathcal T}^c_i)$ or $\mathcal Z = g({\mathcal T}_i,{\mathcal T}^c_i)$ and ${\mathcal Z}^\prime = gy^{-1}({\mathcal T}_i,{\mathcal T}^c_i)$. This implies Item (2) of Lemma \ref{pour le rapporteur II}.

It only remains to check that $({\mathcal T}_i,{\mathcal T}^c_i)$ and $y ({\mathcal T}_i,{\mathcal T}^c_i)$ cross. Obviously (use the previous paragraphs with $g=g_0=y$) $e \in {\mathcal T}_i \cap y {\mathcal T}_i$, $e \in {\mathcal T}_i \cap y {\mathcal T}^c_i$ and $t_i \in {\mathcal T}^c_i \cap y {\mathcal T}_i$ so that  ${\mathcal T}_i \cap y {\mathcal T}_i \neq \emptyset$, ${\mathcal T}_i \cap y {\mathcal T}^c_i \neq \emptyset$ and ${\mathcal T}^c_i \cap y {\mathcal T}_i \neq \emptyset$. In order to prove that ${\mathcal T}^c_i \cap y {\mathcal T}^c_i \neq \emptyset$, let us observe that $t_i \in {\mathcal T}^c_i$ is connected to $y t_i x^{-1}_i t_i$ by the edge-path $(t_i,t_i y) (t_iy,t_iyt_i)$ since $t_i y = y x^{-1}_i t_i$ and $t_i x^{-1}_i = x^{-1}_i t_i$. The edge $(t_i,t_i y)$ is a $y$-edge so belongs to $\Gamma_c \setminus E_i$. The edge $(t_iy,t_iyt_i) = (yt_ix^{-1}_i,yt_ix^{-1}_it_i)$ is in $y E_i$ so, by Claim \ref{49}, does not belong to $E_i$. Hence  $(t_i,t_i y) (t_iy,t_iyt_i)$ is an edge-path in $\Gamma_c \setminus E_i$ from $t_i \in {\mathcal T}^c_i$ to $y t_i x^{-1}_i t_i$ so that $y t_i x^{-1}_i t_i \in {\mathcal T}^c_i$. Moreover $yt_i \in y{\mathcal T}^c_i$ is connected to $yt_ix^{-1}_it_i = yt_i y^{-1}t_iy$ by the edge-path $(yt_i,yt_iy^{-1})(yt_iy^{-1},yt_iy^{-1}t_i)(yt_iy^{-1}t_i,yt_iy^{-1}t_iy)$. The first and last edge in this edge-path are respectively $y^{-1}$- and $y$-edges and so belong to $\Gamma_c \setminus yE_i$. The $t_i$-edge $(yt_iy^{-1},yt_iy^{-1}t_i)$ is in $E_i$ since $yt_iy^{-1} \in H_i$, see Remark \ref{une remarque}. By Claim \ref{49}, it is not in $yE_i$. We so proved that $(yt_i,yt_iy^{-1})(yt_iy^{-1},yt_iy^{-1}t_i)(yt_iy^{-1}t_i,yt_iy^{-1}t_iy)$ is an edge-path from $yt_i \in y{\mathcal T}^c_i$ to $yt_ix^{-1}_it_i$ in $\Gamma_c \setminus yE_i$ so that $yt_ix^{-1}_it_i \in y{\mathcal T}^c_i$. Now $y t_i x^{-1}_i t_i \in {\mathcal T}^c_i$ and $yt_ix^{-1}_it_i \in y{\mathcal T}^c_i$ so that ${\mathcal T}^c_i \cap y {\mathcal T}^c_i \neq \emptyset$ and Item (1) of Lemma \ref{pour le rapporteur II} is proved.
\end{proof}

Let us now conclude the proof of Corollary \ref{allez}. We consider the family $$\mathcal F = \{(\mathcal Y,{\mathcal Y}^c),({\mathcal V}_1,{\mathcal V}^c_1),({\mathcal T}_i,{\mathcal T}^c_i), y ({\mathcal T}_i,{\mathcal T}^c_i) \mbox{  ;  } i=1,2\}$$ of walls of $(G,\mathcal W)$. By Lemma \ref{pour le rapporteur II}, for each $i$ the two vertizontal $i$-walls cross. Let us check the other intersections:

\begin{itemize}
  \item $e \in \mathcal Y \cap {\mathcal V}_1$, $t_1 \in \mathcal Y \cap {\mathcal V}^c_1$, $y \in {\mathcal Y}^c \cap {\mathcal V}_1$ and $yt_1 \in {\mathcal Y}^c \cap {\mathcal V}^c_1$ so that $(\mathcal Y,{\mathcal Y}^c)$ and $({\mathcal V}_1,{\mathcal V}^c_1)$ cross.  
  \item $e \in {\mathcal T}_1 \cap {\mathcal T}_2$, $t_1 \in {\mathcal T}^c_1 \cap {\mathcal T}_2$, $t_2 \in {\mathcal T}_1 \cap {\mathcal T}^c_2$, $t_2t_1 \in {\mathcal T}^c_1 \cap {\mathcal T}^c_2$ so that $({\mathcal T}_1,{\mathcal T}^c_1)$ and $({\mathcal T}_2,{\mathcal T}^c_2)$ cross.
  
  For the intersections ${\mathcal T}^c_i \cap y{\mathcal T}^c_j$ in the following two items, we refer the reader to Figure \ref{derniere figure} (the $t_i$-edges in $E_i$ are the thick edges, the $t_j$-edges in $y E_j$ are the dotted edges). 
  
  \item $e \in {\mathcal T}_1 \cap y{\mathcal T}_2$, $t_1 \in {\mathcal T}^c_1 \cap y{\mathcal T}_2$, $yt_2 \in {\mathcal T}_1 \cap y{\mathcal T}^c_2$, $t_1 y x_2   = y x^{-1}_1t_1 x_2  \in {\mathcal T}^c_1 \cap y{\mathcal T}^c_2$ so that $({\mathcal T}_1,{\mathcal T}^c_1)$ and $y({\mathcal T}_2,{\mathcal T}^c_2)$ cross.
  \item $e \in y{\mathcal T}_1 \cap {\mathcal T}_2$, $yt_1 \in y{\mathcal T}^c_1 \cap {\mathcal T}_2$, $t_2 \in y{\mathcal T}_1 \cap {\mathcal T}^c_2$, $y x^{-1}_2 t_2 x_1 = t_2 y x_1 \in y{\mathcal T}^c_1 \cap {\mathcal T}^c_2$ so that $y({\mathcal T}_1,{\mathcal T}^c_1)$ and $({\mathcal T}_2,{\mathcal T}^c_2)$ cross.
  
  \begin{figure}[htbp]
{\centerline{\includegraphics[height=5cm, viewport = 30 545 598 820,clip]{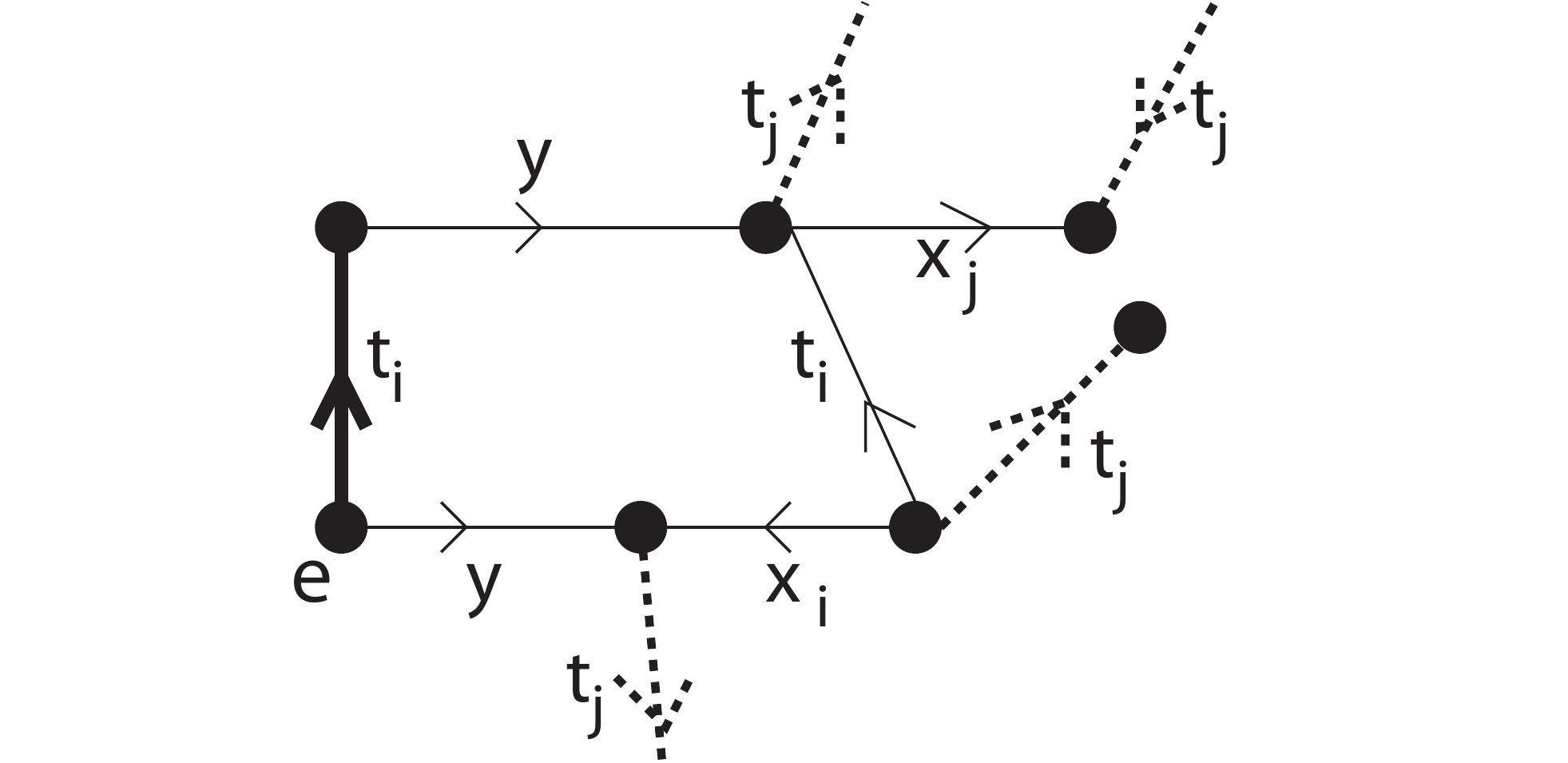}}} \caption{\label{derniere figure}}
\end{figure}
  
  \item for each $i \in \{1,2\}$, $e \in {\mathcal Y} \cap {\mathcal T}_i$, $y \in {\mathcal Y}^c \cap {\mathcal T}_i$, $t_i \in \mathcal Y \cap {\mathcal T}^c_i$, $y t_i y^{-1} t_i \in {\mathcal Y}^c \cap {\mathcal T}^c_i$ so that $(\mathcal Y,{\mathcal Y}^c)$ and $({\mathcal T}_i,{\mathcal T}^c_i)$ cross for each $i \in \{1,2\}$.
  \item for each $i \in \{1,2\}$, $e \in {\mathcal Y} \cap y {\mathcal T}_i$, $y \in {\mathcal Y}^c \cap y {\mathcal T}_i$, $yt_i \in {\mathcal Y}^c \cap y{\mathcal T}^c_i$, $t_i y t_i \in {\mathcal Y} \cap y {\mathcal T}^c_i$ (since $yt_i (y^{-1} t_i y = t_i yx_i (t_i x^{-1}_i y^{-1}y) = t_i y t_i (x^{-1}_i x_i)(y^{-1}y)$) so that $(\mathcal Y,{\mathcal Y}^c)$ and $y({\mathcal T}_i,{\mathcal T}^c_i)$ cross for each $i \in \{1,2\}$.
  \item $e \in {\mathcal V}_1 \cap {\mathcal T}_1$, $x_1 = y^{-1} t^{-1}_1 y t_1 \in {\mathcal V}_1 \cap {\mathcal T}^c_1$, $t_1 \in  {\mathcal V}^c_1 \cap {\mathcal T}^c_1$, $y t_1 y^{-1} = t_1 y x_1y^{-1} \in {\mathcal V}^c_1 \cap {\mathcal T}_1$ so that $({\mathcal V}_1,{\mathcal V}^c_1)$ and $({\mathcal T}_1,{\mathcal T}^c_1)$ cross.
  \item $e \in {\mathcal V}_1 \cap {\mathcal T}_2$, $t_1 \in  {\mathcal V}^c_1 \cap {\mathcal T}_2$, $t_2 \in {\mathcal V}_1 \cap {\mathcal T}^c_2$, $t_1 t_2 \in {\mathcal V}^c_1 \cap {\mathcal T}^c_2$ so that $({\mathcal V}_1,{\mathcal V}^c_1)$ and $({\mathcal T}_2,{\mathcal T}^c_2)$ cross.
  \item $e \in {\mathcal V}_1 \cap y{\mathcal T}_1$, $t_1 \in {\mathcal V}^c_1 \cap y{\mathcal T}_1$, $yt_1 = t_1 y x_1 \in  {\mathcal V}^c_1 \cap y{\mathcal T}^c_1$, $yx_1= y(y^{-1}t^{-1}_1 y)t_1 \in {\mathcal V}_1 \cap y{\mathcal T}^c_1$ so that $({\mathcal V}_1,{\mathcal V}^c_1)$ and $y({\mathcal T}_1,{\mathcal T}^c_1)$ cross.
  \item $e \in {\mathcal V}_1 \cap y{\mathcal T}_2$, $t_1 \in  {\mathcal V}^c_1 \cap y{\mathcal T}_2$, $yt_2 \in {\mathcal V}_1 \cap y{\mathcal T}^c_2$, $yt_1 t_2=t_1yx_1t_2 \in {\mathcal V}^c_1 \cap y{\mathcal T}^c_2$ so that $({\mathcal V}_1,{\mathcal V}^c_1)$ and $y({\mathcal T}_2,{\mathcal T}^c_2)$ cross.
\end{itemize}
  
  Thus the given family $\mathcal F = \{(\mathcal Y,{\mathcal Y}^c),({\mathcal V}_1,{\mathcal V}^c_1),({\mathcal T}_i,{\mathcal T}^c_i), y ({\mathcal T}_i,{\mathcal T}^c_i) \mbox{  ;  } i=1,2\}$ is a family of $6$ pairwise crossing walls of $(G,\mathcal W)$. By Lemmas \ref{pour le rapporteur} and \ref{pour le rapporteur II}, in a family of pairwise crossing walls there are at most one horizontal wall, one vertical wall and two vertizontal $i$-walls for each $i \in \{1,2\}$. Therefore such a family contains at most $6$ distinct walls. The proof of Corollary \ref{allez}, and so of Theorem \ref{the theorem}, at least in the case where $n=2$, is complete. 
\end{proof}

\begin{remark}
As we noticed the generalization to any integer $n \geq 3$
is straightforward: if $\F{n} = \langle t_1,\cdots,t_n \rangle$ denotes the vertical subgroup of $G = \F{n+1} \rtimes_\sigma \F{n}$ then as above the vertical walls are the classical walls of the free group $\F{n}$; if $\F{n+1} = \langle x_1,\cdots,x_n,y \rangle$ denotes the horizontal subgroup of $G$ then as above the horizontal walls are the left $G$-translates of the wall separated by the edges in $\displaystyle \bigcup_{t \in \F{n}} (t,ty)^{\pm 1}$. Finally there is a type of vertizontal wall for each letter in $\{t_1,\cdots,t_{n}\}$ and $i$-vertizontal walls are defined as in \ref{definition enfin} by posing $H_i =  \langle x_{j},t_{j},y x_i y^{-1} t_i, x_i t^{-1}_i  \mbox{  ,  } j \neq i \rangle$. The dimension of the cube complex on which $G$ acts is $2n+2$: the $+2$ comes from the fact that one can always put one horizontal and one vertical wall in a family of pairwise crossing walls, and no more. The $2n$ comes from the fact that there are $n$ distinct types of vertizontal walls and for each vertizontal $i$-wall one can put two distinct $i$-walls in a family of pairwise crossing walls, and no more.
\end{remark}

\begin{remark}
By adapting our construction we get that the group $\F{n+1} \rtimes_\sigma \F{2}$ ($n \geq 3$), where $\sigma(t_i)$ fixes any $x_j$ for $j=1,\cdots,n$ and $\sigma(t_i)(y)=y x_i$, $i=1,2$, acts on a $6$-dimensional CAT(0) cube complex: the walls are the vertical walls defined above, the vertizontal walls associated to $t_i$ defined in a similar way as above (in the definition of the subgroup $H_i$ add $x_3,\cdots,x_n$ as generators) and horizontal walls associated not only to $y$ but also to $x_3,\cdots,x_n$. There are more types of horizontal walls but less types of vertizontal walls than in the ${n}^{\mathrm{th}}$-group of Formanek - Procesi. Since two distinct horizontal walls cannot be in a collection of pairwise crossing walls (contrary to what happens with vertizontal walls) this explains the smaller dimension of the complex in this case.
Thus what is perhaps the most important, for the dimension of the cube complex, is the way the images of the higher edges cover the lower strata. Here the rose with $n+1$ petals is a Bestvina-Feighn-Handel representative. The filtration of the graph is given by $\emptyset \varsubsetneq \{x_1,\cdots,x_n\} \varsubsetneq \{x_1,\cdots,x_n\} \cup \{y\}$. The image of the highest edge $y$ cover $\{ x_1 , x_2 \}$ but not the whole lower stratum $\{x_1,\cdots,x_n \}$ as it is the case when considering the ${n}^{\mathrm{th}}$-group of Formanek - Procesi.
\end{remark}

\noindent {\em Acknowledgements:} The author would like to thank M.~Lustig (Universit\'e d'Aix-Marseille III, Marseille) for giving to him, while working on other topics, the example of free-by-free group we deal with in this paper. It is a pleasure to thank P.A.~ Cherix (Universit\'e de Gen\`eve, Geneva) who introduced the author to the Haagerup property seven years ago.  G.N.~Arzhantseva (Universit\"at Wien, Vienna) and P.A.~Cherix evoked the question of the a-T-menability of free-by-free, and surface-by-free groups. At that time in Geneva, the interest of W.~Pitsch (U.A.B., Barcelona) in this question was a source of motivation. For these reasons, all three of them deserve the gratitude of the author. He is moreover indebted to G.N.~Arzhantseva for telling him about the Question cited in the Introduction. A part of these acknowledgements also goes to G.N.~Arzhantseva, M.~Lustig and Y.~Stalder for their remarks about preliminary versions of the paper. Last but not least, many thanks are due to the referee, who provided a great help to correct some mistakes and gave invaluable comments to make the paper better: among many other things he indicated to the author the results of Chatterji-Niblo \cite{ChatterjiNiblo} and Nica \cite{Nica}, and gave him the idea of looking for the dimension of the cube complex.

\bibliographystyle{plain}
\bibliography{biblioHaagerup}

\end{document}